\title{\large{\textbf{
 A 0-1 LAW FOR THE MASSIVE GAUSSIAN FREE FIELD
 }}}

\date{}

\documentclass[a4paper,11pt, leqno]{article}
\usepackage[english]{babel}
\usepackage[a4paper,left=2.5cm,right=2.5cm,top=2.8cm,bottom=2.8cm]{geometry}
\usepackage{amsfonts, amsmath, amssymb, amsthm, mathrsfs}

\usepackage{layout}
\usepackage{color}

\usepackage{psfrag}
\usepackage{graphicx}
\usepackage[font={footnotesize}]{caption}

\usepackage{comment}
\usepackage[hang,flushmargin]{footmisc}

\numberwithin{equation}{section}

\newtheorem{thm}{Theorem}[section]
\newtheorem{lem}[thm]{Lemma}

\newtheorem{proposition}[thm]{Proposition}

\newtheorem{corollary}[thm]{Corollary}

\theoremstyle{remark}

\theoremstyle{definition}
\newtheorem{remark}[thm]{Remark}

\addtocounter{section}{-1}

\begin{document}

\maketitle

\begin{center}
\vspace{-1cm}
Pierre-Fran\c cois Rodriguez\footnote{\noindent E-mail: pierre.rodriguez@math.ethz.ch. This research was supported in part by the grant ERC-2009-AdG 245728-RWPERCRI.} 

\vspace{1.5cm}
Preliminary draft
\end{center}
\vspace{0.3cm}
\begin{abstract}
\centering
\begin{minipage}{0.9\textwidth}
We investigate the phase transition in a non-planar correlated percolation model with long-range dependence, obtained by considering level sets of a Gaussian free field with mass above a given height $h$. The dependence present in the model is a notorious impediment when trying to analyze the behavior near criticality. Alongside the critical threshold $h_*$ for percolation, a second parameter $h_{**} \geq h_*$ characterizes a strongly subcritical regime.  We prove that the relevant crossing probabilities converge to $1$ polynomially fast below $h_{**}$, which (firmly) suggests that the phase transition is sharp. 
A key tool is the derivation of a suitable differential inequality for the free field that enables the use of a (conditional) influence theorem.


\end{minipage}
\end{abstract}
\thispagestyle{empty}

\vspace{6cm}
\begin{flushleft}
Departement Mathematik \hfill May 2015\\
ETH Z\"urich \\
R\"amistrasse 101\\
CH-8092 Z\"urich \\
Switzerland.
\end{flushleft}


\newpage
\mbox{}
\thispagestyle{empty}
\newpage

\section{Introduction}

Level-set percolation for the (massive and massless) Gaussian free field, whose study goes back at least to Molchanov and Stepanov \cite{MS83a}, as well as Lebowitz and Saleur \cite{LS86}, cf. also \cite{BLM87}, has received renewed and considerable interest in recent times, see for instance \cite{Ga04}, \cite{RoS13}, \cite{PR13}, \cite{DR14}, \cite{Sz14}; see also \cite{Sz12c}, \cite{Lu14}, \cite{Ro14} for links to the model of random interlacements, introduced by Sznitman in the influential work \cite{Sz10}. Albeit the presence of (very strong, in the massless case) correlations with infinite range, the nature of this model opens the door to the rich mathematical theory of Gaussian processes, but in spite of this, its \textit{near-critical regime} is still far from being well understood. In particular, it remains an important open problem to determine whether the phase transition is sharp. As will become apparent shortly, our results yield some progress in this direction. Their proofs rely crucially on the use of a so-called \textit{influence theorem}, which goes back to the seminal works \cite{KKL88} and \cite{BKKKL92}, see also \cite{GG06}, and was arguably (re-)popularized in the context of percolation in \cite{BR06}. Incidentally, let us mention that many celebrated results around this circle of questions for other (correlated) percolation models (see for instance \cite{GG06}, \cite{BDC12}, \cite{DCM14} in the case of random cluster models, and \cite{BR06} for percolation on (random) Voronoi tessellations) have only been proved in the \textit{planar} setting so far (a notable exception being Bernoulli percolation, cf. \cite{Me86}, \cite{AB87}, and \cite{DCT15}), where the aforementioned sharp-threshold techniques are typically paired with duality properties of the lattice to form a powerful set of tools. Our work is to some extent also an attempt to remedy this situation. 

\bigskip

We now describe our results in more detail, and refer to Section \ref{S:NOTATION} for  notation and precise definitions of the various objects involved. We consider the massive Gaussian free field on the Euclidean lattice $\mathbb{Z}^d$, $d\geq 3$, endowed with the usual nearest-neighbor graph structure. Its law $\mathbb{P}_\theta$ on $\mathbb{R}^{\mathbb{Z}^d}$, indexed by a parameter $\theta \in (0,1)$, which will be referred to as the \textit{mass}, is formally given by
\begin{equation}\label{eq:def_GFF_formal}
d\mathbb{P}_\theta \text{ ``$\propto$'' } \exp \Big\{ -\frac12 \big\langle \varphi, \big((1-\theta)\Delta + \theta\big)\varphi \big\rangle_{\ell^2(\mathbb{Z}^d)}\Big\} \prod_x d\lambda(\varphi_x),
\end{equation}
where $\Delta$ denotes the lattice Laplacian (i.e. $(\Delta f)(x) =\frac{1}{2d} \sum_{y: y\sim x}(f(y)-f(x))$, for $f\in \ell^2(\mathbb{Z}^d)$), and $\lambda$ is the Lebesgue measure on $\mathbb{R}$. The precise definition of $\mathbb{P}_\theta$ is given in \eqref{eq:GFgen} below. From a more probabilistic point of view, $\mathbb{P}_\theta$ can be seen as the law of a centered Gaussian field (indexed by $\mathbb{Z}^d$) with covariances given by the Green function of a simple random walk on $\mathbb{Z}^d$, killed uniformly with probability $\theta$ at every step. In particular, this covariance structure decays exponentially fast with distance, see \eqref{eq:st_decay_cov}. Denoting by $\varphi=(\varphi_x)_{x\in \mathbb{Z}^d}$ the canonical field (under the law $\mathbb{P}_\theta$), and for any \textit{level} $h \in \mathbb{R}$, we introduce the random subset of \nolinebreak $\mathbb{Z}^d$
\begin{equation} \label{EQ:INTRO_LEVELSET}
E^{\geq h}  = \{ x \in \mathbb{Z}^d ; \; \varphi_x \geq h \},
\end{equation}
obtained by truncating the field $\varphi$ below height $h$. We will refer to it as the \textit{level set} (above level $h$), and we will study its percolative properties. Note that, as $\theta$ varies, cf. \eqref{eq:def_GFF_formal}, the model interpolates between level-set percolation for the massless free field on the one hand (corresponding to $\theta=0$) and the well-studied case of independent (Bernoulli) site percolation on the other hand (when $\theta=1$), see e.g. the classical reference \cite{Gr99}. Since $E^{\geq h}$ is \textit{decreasing} in \nolinebreak $h$, the corresponding critical parameter is sensibly defined as
\begin{equation}\label{h*}
h_*(\theta, d) = \inf \{ h \in \mathbb{R} \, ; \, \mathbb{P}_{\theta}[0 \stackrel{\geqslant h}{\longleftrightarrow} \infty]=0 \} \ \in [-\infty, \infty]
\end{equation}
(with the convention $\inf \emptyset = \infty$), where $\{0 \stackrel{\geqslant h}{\longleftrightarrow} \infty\}$ is the event that the origin lies in an infinite connected component of $E^{\geq h}$. In particular, $h_*$ is such that, for all $h<h_*$, the set $E^{\geq h}$ contains a (unique) infinite connected component $\mathbb{P}_{\theta}$-a.s. (the \textit{supercritical regime}), whereas for all $h>h_*$,  $E^{\geq h}$ consists $\mathbb{P}_{\theta}$-a.s of finite clusters only (the \textit{subcritical regime}). It has been known at least since the work of \cite{MS83a}, \cite{MS83b} (see also \cite{Ga04}) that there is a non-trivial phase transition in the massive case, i.e.
$$
-\infty < h_*(\theta, d) < \infty, \text{ for all $d\geq 3$ and $\theta \in (0,1)$}
$$
(incidentally, this, and more, is also true when $\theta = 0$, see  \cite{BLM87}, \cite{RoS13}). In fact, mimicking the methods of \cite{RoS13}, as summarized below in Theorem \ref{T:h_**_properties}, one can obtain (much) more quantitative information on the nature of the subcritical regime. To this end, we introduce a central quantity of interest, 
\begin{equation}\label{eq:def_p}
p_{\theta, L}(h) = \mathbb{P}_{\theta}[B(0,L)\stackrel{\geqslant h}{\longleftrightarrow}S(0,2L)], \text{ for } \theta \in (0,1), \, L \geq1, \text{ and } h \in \mathbb{R},
\end{equation}
where the event $\{ B(0,L) \stackrel{\geqslant h}{\longleftrightarrow} S(0, 2L) \}$ refers to the existence of a (nearest-neighbor) path in $E^{\geq h}$ connecting $B(0,L)$, the ball of radius $L$ around $0$ in the $\ell^\infty$-norm, to $S(0, 2L)$, the $\ell^\infty$-sphere of radius $2L$ around $0$.  Note that $p_{\theta, L}(\cdot)$ is a decreasing function. We define a second critical parameter
\begin{equation} \label{eq:h_**}
h_{**}(\theta , d) = \inf \big\{ h \in \mathbb{R} \, ; \; \text{for some $\alpha >0$,} \;  \lim_{L\to \infty} L^\alpha \, p_{\theta, L}(h)= 0 \big\} 
\end{equation}
(the condition on $p_{\theta,L}(\cdot)$ can be somewhat relaxed, see \cite{PR13} and \eqref{eq:h_**bis} below). It is almost immediate that $h_*(\theta,d) \leq h_{**}(\theta, d)$, and one can show (similarly to the massless case, but with some simplifications) that 
\begin{equation*}
h_{**}(\theta, d)<\infty , \text{ for all $d\geq 3$ and $\theta \in (0,1)$}.
\end{equation*}
The threshold $h_{**}$ is a fundamental quantity because it characterizes a \textit{strongly} subcritical regime, in the sense that, for suitable constants $c,c' >0$, $\rho \in (0,1]$, and all $d\geq 3$,
\begin{equation}\label{eq:st_MAINbisbis}
 p_{\theta, L}(h) \leq c(\theta, h) e^{- c'(\theta,h) \cdot L^{\rho(\theta,h)}}, \text{ for all $\theta \in (0,1)$, $h> h_{**}$ and $L \geq 1$}
\end{equation}
(and therefore, the connectivity function $ \mathbb{P}_{\theta}[ 0 \stackrel{\geqslant h}{\longleftrightarrow}  x ]$ decays stretched exponentially in $|x|$, for $h> h_{**}$). It is at present an essential open question to prove (or disprove) that $h_*(\theta,d)$ and $h_{**}(\theta, d)$ coincide. One of the central results of this paper is the following theorem. Roughly speaking, it establishes an ``approximate 0--1 law'' for the function $ p_{\theta, L}(\cdot)$ around $h_{**}$ by providing a suitable lower bound, companion to \eqref{eq:st_MAINbisbis}, for the probability $p_{\theta, L}(h)$ at values of $h$ slightly \textit{below} $h_{**}$. 

\begin{thm}\label{T:MAIN}$(d\geq 3)$

\medskip
\noindent For all $\theta \in (0,1)$ and $h< h_{**}(\theta)$, there exist constants $\varepsilon(\theta, h) > 0$ and $C_0(\theta,h) \geq 1$ such that, for all $L \geq 1$,
\begin{equation}\label{EQ:MAIN}
p_{\theta, L}(h) \geq 1- C_0(\theta,h) \cdot L^{-\varepsilon(\theta, h)}. 
\end{equation}

\end{thm}

Theorem \ref{T:MAIN} can be somewhat generalized, see Remark \ref{R:FINAL}, 2) below. In particular, it also holds mutatis mutandis in dimension two. We deliberately refrain from including this case in our exposition because a few results shown along the way continue to hold in the massless case (which is not well-defined when $d=2$). Moreover, as hinted at in the expository paragraph, we wish to emphasize that our methods are completely non-planar.

We now give a broad outline of the proof. It would be too reductive to let Theorem \ref{T:MAIN} stand alone, as some of the results deduced \textit{en route}, and among them, certain differential inequalities, are interesting in their own right. The polynomial speed of convergence in \eqref{EQ:MAIN} is ultimately obtained by working under periodic boundary conditions, hence most of the results presented below hold more generally under $\mathbb{P}_\theta^{\mathcal{G}}$, where $\mathcal{G}$ can either be a discrete torus or $\mathbb{Z}^d$ (so that $\mathbb{P}_\theta^{\mathbb{Z}^d} = \mathbb{P}_\theta$), cf. \eqref{eq:GFgen}, \eqref{EQ:notation} for precise definitions. We investigate the derivative $d \,  \mathbb{P}_{\theta}^{\mathcal{G}}[A^{h}] / dh$, for arbitrary $A \subset \{ 0,1\}^{K}$, with $K \subset \subset \mathcal{G}$ and $h \in \mathbb{R}$ (where $A^h = \{ (1\{ \varphi_x \geq h\})_x \in A \}$, part of $\mathbb{R}^{\mathcal{G}}$). First, we arrive in Proposition \ref{P:Russo_formula} to a Margulis-Russo type formula, cf. \cite{Ma74}, \cite{Ru81}, for the free field (valid also when $\theta=0$), of the form
\begin{equation}\label{eq:Russo_intuitive}
- \frac{d \,  \mathbb{P}_{\theta}^{\mathcal{G}}[A^{h}] }{dh} =\sum_{x\in K} T_{\theta}(A^h,x)
\end{equation}
(the minus sign is because $ \mathbb{P}_{\theta}^{\mathcal{G}}[A^{h}]$ is decreasing in $h$, when $A$ is increasing, by our above convention, cf. \eqref{EQ:INTRO_LEVELSET}). The specific form of the terms $T_{\theta}(A^h,x)$ is irrelevant for the purposes of this Introduction, but we note the following that $T_{\theta}(A^h,x) \geq 0$ for all $x\in K$, $h \in \mathbb{R}$, whenever $A$ is increasing, cf. Remark \ref{R:deriv_comments}, 1). Moreover, the sum appearing in \eqref{eq:Russo_intuitive} can naturally be thought of as arising from a chain rule when computing the derivative.

A key point, which comes in Proposition \ref{P:dom_inf_ptwise} below, is that, under certain assumptions, the right-hand side of \eqref{eq:Russo_intuitive} can be made to ``dominate'' a sum of so-called conditional influences, defined as
\begin{equation}\label{eq:inf_00}
I^{\mathcal{G}}_{\theta} (A^h, x) = \mathbb{P}_{\theta}^{\mathcal{G}}[A^{h}\, | \, \varphi_x \geq h] - \mathbb{P}_{\theta}^{\mathcal{G}}[A^{h}\, | \, \varphi_x < h],
\end{equation}
and introduced by Graham and Grimmett in \cite{GG06} in the context of the random-cluster model. A corresponding influence theorem is derived in  \cite{GG06}, cf. also Theorem \ref{T:inf_thm} below, and essentially obtained by pairing a strong FKG property (which holds in the present case as well, see Lemma \ref{L:FKG1}) with the classical results of \cite{KKL88}, \cite{BKKKL92}. More precisely, Proposition \ref{P:dom_inf_ptwise} implies that, given $M>0$, for all increasing events $A\subset \{ 0,1\}^K$ and $h \in (-M, M)$,
\begin{equation}\label{0000}
- \frac{d \,  \mathbb{P}_{\theta}^{\mathcal{G}}[A^{h}] }{dh}  \geq {c}(\theta, M) \sum_{x\in K} I_{\theta}(A^h,x).
\end{equation}
It is crucial that the constant $c$ appearing here is uniform in $K$. In fact, rather than the required comparison ``in $\ell^1$,'' we provide an (arguably much stronger) \textit{pointwise} estimate of the form $T_{\theta}(A^h,x) \geq c(\theta, M) \cdot I_{\theta}(A^h,x)$ for $h\in  (-M, M)$. 

The domination result \eqref{0000} is then paired with an influence theorem to yield the aforementioned differential inequality for $\mathbb{P}_{\theta}^{\mathcal{G}}[A^{h}] $, when $A$ is increasing, see Corollary \ref{C:diff_ineq}. 
We refrain from writing it down explicitly here, but rather note that, in order to prove a statement such as \eqref{EQ:MAIN}, one typically would like the derivative $- d \,  \mathbb{P}_{\theta}^{\mathcal{G}}[A^{h}] /dh$ to be sufficiently \textit{large}, as to yield a meaningful lower bound on $\mathbb{P}_{\theta}^{\mathcal{G}}[A^{h}]$ upon integration over the interval of interest (in our case, located slightly below $h_{**}$). As it turns out, this requires showing that the maximal influence is sufficiently small. In Proposition \ref{P:NO_SPEED} below, we prove a weaker version of Theorem \ref{T:MAIN} (without speed of convergence), by working directly on $\mathcal{G}=\mathbb{Z}^d$, and proving (see Lemma \ref{L:inf_bound_BIS}) a suitable upper bound for the maximal influence of the box-to-box crossing events of \eqref{eq:def_p}. We won't discuss the details of the proof here, but the bottom line is that this is due to the geometry of the event $\{ B(0,L) \stackrel{\geqslant h}{\longleftrightarrow} S(0, 2L) \}$ in question (for comparison, consider an event of the type $\{ 0 \stackrel{\geqslant h}{\longleftrightarrow} S(0,L)\}$, where the points around the origin are expected to have a rather large influence). The actual proof of Theorem \ref{T:MAIN} is presented thereafter, and bypasses this necessity by working with a suitably chosen \textit{translation invariant} event $\mathcal{A}_L^h$, under periodic boundary conditions (inside a torus $\mathbb{T}_{\overline{L}}$ of ``size'' $\overline{L} \gtrsim L$). Indeed, the influence theorem automatically gives a \textit{lower bound} on the maximal influence, cf.  \eqref{eq:inf_thm_l_infinity}, which is a priori of little help, but becomes relevant when considering translation invariant events (since all influences are then equal by symmetry). The definition of $\mathcal{A}_L^h$ allows for the resulting lower bound on $\mathbb{P}_{\theta}^{\mathbb{T}_{\overline{L}}}[\mathcal{A}_L^h]$, for $h$ close to $h_{**}$, obtained from the (periodic) differential inequality, to be translated back to a meaningful lower bound on $p_{\theta,L}(h)$, thus yielding \eqref{EQ:MAIN}.

\bigskip

We now describe the organization of this article. Section \ref{S:NOTATION} introduces some notation as well as the main objects involved, and recalls certain properties of the free field that will be used repeatedly in the sequel. It also contains a proof of the ``FKG lattice condition'' needed for the application of the influence theorem. Section \ref{S:PT_RUSSO} briefly reviews some known results regarding the phase transition, and Proposition \ref{P:Russo_formula} contains the ``Margulis-Russo''-type formula mentioned above. Section \ref{S:dom_inf} is centered around the proof of \eqref{0000}, which is the object of Proposition \ref{P:dom_inf_ptwise}. This Proposition is then combined with an influence theorem (Theorem \ref{T:inf_thm}) to yield the desired differential inequalities in Corollary \ref{C:diff_ineq}. Finally, Section \ref{S:APPS} deals with the applications to the crossing events of interest, and contains in particular the proof of Theorem \ref{T:MAIN}.

\bigskip

A word about constants: in what follows $c,c', c'', \dots$ denote positive constants having values that can change from place to place. Numbered constants $c_0,c_1,c_2,\dots$ are defined upon first appearance in the text and remain fixed from then on until the end of the article. The dependence of constants on the dimension $d$ will be kept implicit throughout, but their dependence on any other parameter will appear explicitly in the notation.

\section{Notation and Preliminaries}\label{S:NOTATION}

In this section, we introduce some notation to be used in the sequel, as well as the random walks and Gaussian fields of interest. We also collect some of their properties which will be of importance below. These include in particular a brief reminder on conditional expectations for the free field, and a certain (strong) FKG-type inequality.

We denote by 
$\mathbb{Z}=\{\dots,-1,0,1,\dots\}$ the set of integers, write $\mathbb{R}$ for the set of real numbers, and abbreviate $x \wedge y = \min \{x,y\}$ and $x \lor y = \max\{ x,y\}$ for any two numbers $x,y \in \mathbb{R}$. 
We consider the lattice $\mathbb{Z}^{d}$ (tacitly assuming throughout
that $d\geq3$) or the discrete torus $\mathbb{T}_L = (\mathbb{Z}/ 2L\mathbb{Z})^d$, for some $L \geq 1$ and $d\geq 3$ (the extra factor of $2$ is for later convenience), which we endow with the usual nearest-neighbor graph
structure, and denote by $\vert\cdot\vert$ the $\ell^{\infty}$ distance on it. In what follows, unless specified otherwise, the vertex set $\mathcal{G}$ stands for either $\mathbb{Z}^{d}$ or $\mathbb{T}_L$, $L \geq 1$. We will often use $x \sim y$ instead of $|x-y|_1=1$ (where $|\cdot|_1$ denotes the $\ell^1$, i.e. graph distance) for two neighboring vertices $x,y \in \mathcal{G}$. Moreover, for $r\geq0$ and $x\in\mathcal{G}$, we let $B(x,r)=\{y \in \mathcal{G};\ \vert y-x\vert\leq r\}$ and $S(x,r)=\{y\in\mathcal{G};\ \vert y-x\vert=r\}$ stand for
the the $\ell^{\infty}$-ball and $\ell^{\infty}$-sphere of radius $r$ centered
at $x$, and simply write $B_{r}$ and $S_{r}$ if $x=0$ (in the case of $\mathbb{Z}^d$). Given $K$
and $U$ subsets of $\mathcal{G}$, $K^{c}=\mathcal{G}\setminus K$
stands for the complement of $K$ in $\mathcal{G}$, $\vert K\vert$
for the cardinality of $K$, and $K\subset\subset\mathbb{Z}^{d}$
means that $\vert K\vert<\infty$. 

We now introduce the random walks of interest. To
this end, we add a cemetery state $\Delta$ to $\mathcal{G}$ ($=\mathbb{Z}^{d}$ or $\mathbb{T}_L $), i.e. we connect each vertex in $\mathcal{G} \cup \{ \Delta \}$ by an edge to $\Delta$ and denote by $W$ the space of nearest-neighbor $(\mathcal{G} \cup\{\Delta\})$-valued trajectories defined for non-negative times which are absorbed in
$\Delta$ once they reach it, i.e. of sequences $(x_{n})_{n\geq0}$
satisfying $x_{n}\in\mathcal{G} \cup\{\Delta\}$, with $x_{n+1} \sim x_{n}$
for all $n \geq 0 $, $x_{n} =\Delta$ for some $n \geq 0$, and $x_{k+1}=\Delta$ whenever $x_{k}=\Delta$, for some $k\geq0$. We let $\mathcal{W}$,
$(X_{n})_{n\geq0}$, stand for the canonical $\sigma$-algebra and
canonical process on $W$, respectively. Given a parameter $\theta\in[0,1]$,
we consider the Markov chain on $\mathcal{G}\cup\{\Delta\}$ with
transition probabilities 
$$
p_{x,y}=\frac{1}{2d}(1-\theta)1_{\{x \sim y\}},\quad p_{x,\Delta}=\theta,\quad p_{\Delta,\Delta}=1,\quad\text{for all \ensuremath{x,y\in \mathcal{G}}}.
$$

We will refer to $\theta$ as the \textit{mass} of the system. We denote by $P_{\theta}^{x}$ the canonical law on $(W,\mathcal{W})$ of the walk starting at $x\in\mathbb{Z}^{d}$, and by $E_{\theta}^{x}$
the corresponding expectation. 
Thus, $P_{\theta}^{x}$ describes a random walk on $\mathcal{G}$ which is killed uniformly with probability $\theta$ at every step. Somewhat more generally, given a subset $U\subset\mathcal{G}$, we write
$P_{\theta,U}^{x}$ for the law (on $(W,\mathcal{W})$) of the walk
starting at $x\in\mathcal{G}$ killed uniformly at rate $\theta$
or when first entering in $U$ (in particular, we have $P_{\theta,\emptyset}^{x} = P_{\theta}^{x}$). 
Accordingly, we introduce the Green function $g_{\theta,U}(\cdot,\cdot)$ of
this walk as 
\begin{equation}
\begin{split}g_{\theta,U}(x,y) & =\sum_{n\geq0}P_{\theta,U}^{x}[X_{n}=y] =\sum_{n\geq0}(1-\theta)^{n}P^x_{0}[X_{n}=y,\ n<H_{U}], \text{  for }x,y\in\mathcal{G},
\end{split}
\label{EQ:st_GreenFunction}
\end{equation}
where $H_{U}=\inf\{n\geq0;X_{n}\in U\}$ denotes the entrance time in $U$. We will also need the stopping time $\widetilde{H}_{U}=\inf\{n\geq1;X_{n}\in U\}$, the hitting time of $U$. Note that $g_{\theta,U}(x,y)$ is finite and symmetric in both coordinates, and vanishes if $x\in U$ or $y\in U$. We simply write $g_{\theta}(\cdot,\cdot)$ when $U=\emptyset$,
and observe that $g_{\theta}(x,y)=g_{\theta}(x-y,0)\stackrel{\text{def.}}{=}g_{\theta}(x-y)$ due to translation invariance. Moreover, for all $U \subseteq \mathcal{G}$ and $K \subseteq U^c$, we have, by the strong Markov property (at time $H_{K}$), 
\begin{equation}
g_{\theta, U}(x,y)=g_{\theta, U\cup K}(x,y)+E_{\theta,U}^{x}[H_{K}<\infty,\, g_{\theta,U}(X_{H_{K}},y)],\qquad\text{for }x,y\in\mathcal{G}.\label{G-GsubK}
\end{equation}
Finally, it follows immediately from \eqref{EQ:st_GreenFunction} that
\begin{equation}\label{eq:st_decay_cov}
g_{\theta,U}(x,y) \leq c(\theta) e^{-c'(\theta)\cdot |x-y|}, \text{ for all $U\subset \mathcal{G}$, $x,y \in \mathcal{G}$}.
\end{equation}
For future reference, we also note that the entrance probability in $K$ can
be expressed as 
\begin{equation}
P_{\theta}^{x}[H_{K}<\infty]=\sum_{n\geq0}\sum_{y\in K}P_{\theta}^{x}[X_{n}=y,\widetilde{H}_{K}\circ\tau_{n}=\infty]=\sum_{y\in K}g_{\theta}(x,y)\cdot P_{\theta}^{y}[\widetilde{H}_{K}=\infty] ,\label{EQ:st_entr_prob}
\end{equation}
for all $x\in\mathcal{G}$, where $\tau_{n}w(k)=w(k+n)$, for $n,k\geq0$,
$w\in W$ denote the canonical shifts and the last step follows from
the simple Markov property (at time $n$).

Next, we define the Dirichlet forms associated to the above random walks. For arbitrary $f \in \ell^2(\mathcal{G})$, we let
\begin{equation}\label{eq:Dirichlet}
\mathcal{E}(f,f)=\frac{1}{2}\sum_{ x\sim y} \frac{1-\theta}{2d}(f(y)-f(x))^2+ \theta \cdot\sum_{x}f(x)^2,
\end{equation}
where the sums run over $x,y \in \mathcal{G}$. The quantity $\mathcal{E}(f,f) \geq 0$ is finite, for all
$f \in \ell^2(\mathcal{G})$, and can be extended to a bilinear form $\mathcal{E}(\cdot,\cdot)$ on $\ell^2(\mathcal{G}) \times \ell^2(\mathcal{G})$ by polarization. Given $K \subset \subset \mathcal{G}$, we also define the trace Dirichlet form on $K$,
\begin{equation}\label{eq:Trace_Dirichlet}
\mathcal{E}^{\text{tr}}_K(f,f)= \frac{1}{2} \sum_{ x,y \in K} c^{\text{tr}}_{x,y} (f(y)-f(x))^2+ \sum_{x \in K} \kappa^{\text{tr}}_{x}f(x)^2,
\end{equation}
for $f \in \mathbb{R}^K$, where 
\begin{align}
&c^{\text{tr}}_{x,y} = P_{\theta}^{x}[\widetilde{H}_K < \infty, X_{\widetilde{H}_K}=y]1_{\{x \neq y \}}, \text{ for $x, y \in K$} \label{eq:trace_cond}\\
&\kappa^{\text{tr}}_{x} = P_{\theta}^{x}[\widetilde{H}_K = \infty], \text{ for $x \in K$} \label{eq:trace_kill}
\end{align}
(in particular, $\mathcal{E}^{\text{tr}}_{\mathcal{G}}(\cdot,\cdot) = \mathcal{E}(\cdot,\cdot)$).

We now introduce the Gaussian fields of interest. For $\theta\in(0,1]$ and $U\subseteq \mathcal{G}$, we denote by $\mathbb{P}^{\, \mathcal{G}}_{\theta, U}$ the law on $ \mathbb{R}^{\mathcal{G}}$, endowed with its canonical $\sigma$-algebra, under which the canonical coordinates $\varphi=(\varphi_{x})_{x\in\mathcal{G}}$ are distributed as a centered Gaussian field with covariance 
\begin{equation}\label{eq:GFgen}
\mathbb{E}^{\, \mathcal{G}}_{\theta,U}[\varphi_{x}\varphi_{y}]=g_{\theta,U}(x,y),\quad\text{for all \ensuremath{x,y\in\mathcal{G}}},
\end{equation}
and simply write $\mathbb{P}^{\, \mathcal{G}}_{\theta}$ when $U=\emptyset$.
In particular $\varphi_{x}=0$, $\mathbb{P}^{\, \mathcal{G}}_{\theta,U}$-a.s. for
every $x\in U$. 
Note that the massless case is excluded here, since the measure $\mathbb{P}^{\, \mathcal{G}}_{\theta= 0}$ is not well-defined in the case of periodic boundary conditions. Although our main theorem regards positive mass only, some of the results we show along the way also hold for $\mathbb{P}^{\, \mathbb{Z}^d}_{0}$. Thus, we adopt the following convention. In writing a statement concerning some generic measure $\mathbb{P}^{\, \mathcal{G}}_{\theta, U}$ with $U \subset \mathcal{G}$ and $\theta \in [0,1]$, we \textit{always} tacitly assume that $\theta =0$ is excluded in case $\mathcal{G} = \mathbb{T}_L$ for some $L \geq 1$.

We proceed by recalling a classical fact concerning conditional distributions for the Gaussian free field. 
\begin{lem} \label{L:cond_exps} $(\theta \in [0,1], \, U \subset \mathcal{G},\,  K \subset \subset U^c)$
\medskip

\noindent Let $(\widetilde{\varphi}_x)_{x \in \mathcal{G}}$ be defined by
\vspace{-0.5ex}
\begin{equation} \label{mudecomp}
\varphi_x = \widetilde{\varphi}_x + \mu_x,  \text{ for } x \in \mathcal{G},
\end{equation}
where $\mu_x$ is the $\sigma(\varphi_x ; x \in K)$-measurable map defined as
\begin{equation} \label{mu}
\mu_x = E^x_{\theta,U} [H_{K} < \infty , \varphi_{X_{H_{K}}}] = \sum_{y \in K } P^x_{\theta,U} [H_{K} < \infty , X_{H_{K}} = y] \cdot \varphi_y,  \text{ for } x \in \mathcal{G}.
\end{equation}
Then,
\begin{equation} \label{ind+cond_exps}
\begin{split}
&\text{under $\mathbb{P}^{\, \mathcal{G}}_{\theta,U}$, the field $(\widetilde{\varphi}_x)_{x \in \mathcal{G}} $  is independent from} \\
&\text{$\sigma(\varphi_x ; x \in K)$, and distributed as $(\varphi_x)_{x \in \mathcal{G}}$ under $\mathbb{P}^{\, \mathcal{G}}_{\theta, U\cup K}$.}
\end{split}
\end{equation}
In addition, the law of $(\varphi_x)_{x\in K}$ under $\mathbb{P}^{\, \mathcal{G}}_{\theta}$ is given by
\begin{equation}\label{eq:phi_K_law}
\frac{1}{Z(\mathcal{G}, K, \theta)} \exp \Big[  -\frac12 \mathcal{E}^{\mathrm{tr}}_K(\varphi, \varphi) \Big] \prod_{x\in K} d\lambda (\varphi_x),
\end{equation}
where $\mathcal{E}^{\mathrm{tr}}_K(\cdot, \cdot)$ refers to the trace Dirichlet form defined in \eqref{eq:Trace_Dirichlet}, $\lambda$ denotes Lebesgue measure on $\mathbb{R}$ and $Z$ is a suitable normalizing constant.
\end{lem}

\begin{proof}
For $\mathcal{G}= \mathbb{T}_L$, with $L \geq 1$ (and $\theta >0$), this follows immediately from Proposition 2.3 of \cite{Sz12b}. For $\mathcal{G}= \mathbb{Z}^d$, one first considers the measure $\mathbb{P}^{\, \mathcal{G}}_{\theta, U \cup \Lambda^c}$ instead of $\mathbb{P}^{\, \mathcal{G}}_{\theta, U}$, with $\Lambda$ a large box, to which Proposition 2.3 of \cite{Sz12b} applies, and then lets $\Lambda \nearrow \mathbb{Z}^d$. We refer the Reader to the proof of Lemma 1.2 in \cite{RoS13}, which provides the details for $\theta = 0$ and $U = \emptyset$. The case of positive mass or $U \neq \emptyset$ is completely analogous.
\end{proof}

\begin{remark} \label{Rcond_exps}$(\theta \in [0,1], \, U \subset \mathcal{G})$
\smallskip 

\noindent Using Lemma \ref{L:cond_exps}, one obtains the following choice of regular conditional distributions for $(\varphi_{x})_{x\in\mathcal{G}}$ under $\mathbb{P}_{\theta, U}^{\mathcal{G}}$ conditioned on the variables $(\varphi_{x})_{x\in K}$, for some $K\subset\subset\mathcal{G} \cap U^c$, which will prove very useful in many instances below. Namely, $\mathbb{P}_{\theta, U}^{\mathcal{G}}$-a.s.,
\begin{equation}
\mathbb{P}_{\theta, U}^{\mathcal{G}}[(\varphi_{x})_{x\in\mathcal{G}}\in\cdot\,\vert\,(\varphi_{x})_{x\in K}]=\widetilde{\mathbb{P}}_{\theta, U\cup K}^{\mathcal{G}}[(\widetilde{\varphi}_{x}+\mu_{x})_{x\in\mathcal{G}}\in\cdot\,],\label{EQ:st_phi_cond_exps}
\end{equation}
where the field $(\widetilde{\varphi}_{x})_{x\in\mathcal{G}}$ is independent of $\varphi$ under $\widetilde{\mathbb{P}}_{\theta, U\cup K}^{\mathcal{G}}$ (a copy of $\mathbb{P}_{\theta, U\cup K}^{\mathcal{G}}$), and with $(\mu_{x})_{x\in\mathcal{G}}$ as defined in \eqref{mu}. In particular, with \eqref{EQ:st_phi_cond_exps} at hand, note that
\begin{equation} \label{eq:cond_0}
\mathbb{P}_{\theta, U}^{\mathcal{G}}[\, \cdot\,\vert\,\varphi_{x}=0 ,\, x\in K] \stackrel{\text{law}}{=} {\mathbb{P}}_{\theta, U\cup K}^{\mathcal{G}} [\, \cdot\,].
\end{equation}
\hfill $\square$
\end{remark}

Next, we introduce a specific class of events pertaining to level sets of $\varphi$. Given a height profile $\mathbf{h}= (h_x)_{x \in \mathcal{G}} \in \mathbb{R}^{\mathcal{G}}$, we define the occupation field
\begin{equation}\label{eq:def_conf}
\xi^{\mathbf{h}} \equiv \xi^{\mathbf{h}} (\varphi)= (\xi^{\mathbf{h}}_x)_{x\in \mathcal{G}}, \text{ where } \xi^{\mathbf{h}}_x = 1\{ \varphi_x \geq h_x\}.
\end{equation}
We let $\Omega = \{ 0,1\}$ and endow the space $\Omega^{\mathcal{G}} = \{ 0,1 \}^{\mathcal{G}}$ with its canonical $\sigma$-algebra, and canonical coordinates $Y_x$, $x\in \mathcal{G}$. For any measurable set $A \subset \Omega^{\mathcal{G}}$, we define
\begin{equation}\label{eq:events_levset}
A^{\mathbf{h}} \equiv A^{\mathbf{h}} (\varphi) =  \{\xi^{\mathbf{h}} \in A \} \text{ (part of $\mathbb{R}^{\mathcal{G}}$)}.
\end{equation}
With a slight abuse of notation, if $A\subset \Omega^{K}$ for some $K \subset \mathcal{G}$, we write $A^{\mathbf{h}}$ for the event $\{\xi^{\mathbf{h}}_{\vert_K} \in A \} \, (\subset \mathbb{R}^{\mathcal{G}})$. We will often encounter the following types of crossing events. Given $A,B \subset \mathcal{G}$, the event $\{ A \leftrightarrow B\} \subset \Omega^{\mathcal{G}}$ refers to the existence of a nearest-neighbor path of $1$'s connecting the sets $A$ and $B$, and we let
\begin{equation}\label{eq:events_crossing}
\{ A\stackrel{\geqslant \mathbf{h}}{\longleftrightarrow} B\} =  \{\xi^{\mathbf{h}} \in \{ A \leftrightarrow B\} \}.
\end{equation}
We simply write $A^h$, $\{ A\stackrel{\geq h}{\longleftrightarrow} B\}$, in \eqref{eq:events_levset}, \eqref{eq:events_crossing}, if $h_x = h$ for some $h\in \mathbb{R}$ and all $x \in \mathcal{G}$. Given a configuration $\omega \in \Omega^{\mathcal{G}}$ and $K \subset \subset \mathcal{G}$, we write $\omega^K$ resp. $\omega_K$ for the configuration obtained by replacing $\omega(x)$ by $1$, resp. $0$, for all $x \in K$. We simply write $\omega^x$ and $\omega_x$ when $K =\{ x \}$ is a singleton. Finally, we abbreviate by $\{ \xi^{\mathbf{h}}= \omega \text{ on } K \}$ the event $\{\xi^{\mathbf{h}}(x)= \omega(x) \text{ for all } x\in K \}$.

We proceed by discussing an FKG-type inequality obtained when conditioning on a specific configuration of the level set in some finite region of the lattice, which manifests the positive association inherent to the free field. 
We recall the following definition: given two (probability) measures $\mu, \nu$ on a common partially ordered measure space, $\mu$ is said to \textit{stochastically dominate} $\nu$, written as $\mu \geq_{\text{st.}} \nu$, if  $E^{\mu}[f] \geq E^{\nu}[f]$ for all increasing and integrable 
(with respect to both $\mu$ and $\nu$) functions $f$.

\begin{lem} $(\theta \in [0,1], \, U, K \subset \subset \mathcal{G},\, U \cap K = \emptyset, \, \omega \in \Omega^{\mathcal{G}},  \mathbf{h} \in \mathbb{R}^{\mathcal{G}})$

\label{L:FKG1}
\medskip
\noindent 
\begin{align}
& \mathbb{P}_{\theta, U}^{\mathcal{G}}[\,  \cdot \, | \, \xi^{\mathbf{h}}= \omega \text{ on } K] \leq_{\mathrm{st.}}  \mathbb{P}_{\theta, U}^{\mathcal{G}}[\, \cdot \, | \, \xi^{\mathbf{h}}= \omega^x \text{ on } K], \text{ for all }  \, x \in K \label{FKG1}. 
\end{align}
\end{lem}

\begin{proof}
We consider the case $\mathcal{G}= \mathbb{Z}^d$, which is slightly more involved. 
The necessary small alterations needed in the periodic case are indicated at the end. For the sake of clarity, we omit $\mathcal{G}$ and $\theta$ from the notation. First, we reduce \eqref{FKG1} to a similar statement involving (finite-volume) Gibbs measures. To this end, let $n_0 \geq 1$ be large enough so that $B(0,n_0) \supset  K$ and define $U_n = U \cup B(0,n)^c$, for $n \geq n_0$. It is easy to see, using dominated convergence, that $\mathbb{P}_{U_n} \stackrel{w}{\to} \mathbb{P}_U$ as $n \to \infty$. Let us abbreviate $\mathbb{P}_n^+ =  \mathbb{P}_{U_n}[\,  \cdot \, | \, \xi^{\mathbf{h}}= \omega^x \text{ on } K]$ and $\mathbb{P}_n^- =  \mathbb{P}_{U_n}[\,  \cdot \, | \, \xi^{\mathbf{h}}= \omega_x \text{ on } K]$. Since stochastic domination is preserved under weak limits, it suffices to prove that $\mathbb{P}_n^- \leq_{\text{st.}} \mathbb{P}_n^+$, for all $n \geq n_0$. Define the potentials $V_{\pm}: \mathbb{R} \to \mathbb{R}$ as
$$
V_+ (t) = |t| \cdot 1\{ t< 0\}, \ V_-(t) = V_+(-t), \text{ $t\in \mathbb{R}$}.
$$
Note that
\begin{equation}\label{eq:pot}
\text{$V_+ (\cdot)$ is decreasing, continuous 
and $\lim_{\lambda \to \infty } \lambda \cdot V_+(t) = \infty \cdot 1\{t<0\}$, for all $t\in\mathbb{R}$}
\end{equation}
(in fact, any function $V_+$ with these properties would work). Let $\omega$ be fixed, and $\mathscr{C}_K = \mathscr{C}_K(\omega)$ (resp. $\mathscr{O}_K$) denote the closed (resp. open) sites of $K$ in the configuration $\omega$. We may assume that $x$ appearing in \eqref{FKG1} belongs to $\mathscr{C}_K$, else there is nothing to prove. We consider the Hamiltonians
\begin{equation} \label{eq:Hams}
\begin{split}
&H^{\pm}_{n,\lambda}(\varphi) = \mathcal{E}^{\text{tr}}_{V_n}(\varphi,\varphi) + \lambda \Big[ \sum_{y\in \mathscr{O}_K} V_+(\varphi_y - h_y) +  \sum_{z\in \mathscr{C}_K\setminus\{x\}} V_-(\varphi_z - h_z)+ V_{\pm}(\varphi_x - h_x)\Big] ,
\end{split}
\end{equation}
for $\varphi \in \mathbb{R}^{V_n}$, $V_n = U_n^c$, and $\lambda >0$, where $\mathcal{E}^{\text{tr}}_{V_n}(\cdot,\cdot)$ denotes the trace Dirichlet form, cf. \eqref{eq:Trace_Dirichlet}. We define the (Gibbs) measures
\begin{equation}\label{eq:gibbs_m}
d\mathbb{P}_{n,\lambda}^{\pm} \stackrel{\text{def.}}{=}\frac{1}{Z^{\pm}_{n,\lambda}} \exp \Big[ - \frac{1}{2} H^{\pm}_{n,\lambda}(\varphi) \Big] dl (\varphi), \text{ for $\lambda > 0, \, n \geq n_0$},
\end{equation}
where $Z^{\pm}_{n,\lambda} = \int_{\mathbb{R}^{V_n}} \exp[  - \frac12 H^{\pm}_{n,\lambda}(\varphi) ] dl (\varphi)$ is a (finite) normalizing constant and $l$ denotes Lebesgue measure on $\mathbb{R}^{V_n}$. It is then easy to show, using \eqref{eq:phi_K_law}, \eqref{eq:pot} and \eqref{eq:Hams}, that $\mathbb{P}_{n,\lambda}^{\pm} \stackrel{w}{\rightarrow} \mathbb{P}_n^{\pm}$  as $\lambda \to \infty$, for all $n \geq n_0$. Hence, the proof of \eqref{FKG1} reduces to showing that
\begin{equation}\label{eq:FKG_red}
\mathbb{P}_{n,\lambda}^{-}  \leq_{\text{st.}} \mathbb{P}_{n,\lambda}^{+}, \text{ for all $\lambda > 0$ and $n \geq n_0$.}
\end{equation}
A classical result of Holley \cite{Ho74} (in fact, we use its generalization to continuous distributions by Preston, see \cite{Pr74}, Theorem 3) applied to the measures $\mathbb{P}_{n,\lambda}^{\pm}$ yields, by means of  \eqref{eq:gibbs_m}, that in order to prove \eqref{eq:FKG_red}, it suffices to show 
\begin{equation} \label{eq:FKG_red2}
H^{+}_{n,\lambda}(\varphi \vee \varphi') + H^{-}_{n,\lambda}(\varphi \wedge \varphi')\leq H^{+}_{n,\lambda}(\varphi ) + H^{-}_{n,\lambda}(\varphi'),
\end{equation}
for all $\lambda > 0$, $n \geq n_0$ and $\varphi, \varphi' \in \mathbb{R}^{V_n}$. One verifies \eqref{eq:FKG_red2} by expanding all the gradient terms $(\varphi_y - \varphi_z)^2$, with $y \neq z \in V_n$ in the Dirichlet form $ \mathcal{E}^{\text{tr}}_{V_n}(\varphi,\varphi)$, cf. \eqref{eq:Trace_Dirichlet}, and checking \eqref{eq:FKG_red2} for each summand appearing in $H^{\pm}_{n,\lambda}(\varphi)$ \textit{individually} as follows: for the interaction terms proportional to $\varphi_y\varphi_z$, $y \neq z$, one applies the elementary inequality, valid for all $a,a' b, b' \in \mathbb{R}$,
$$
(a \vee a')(b\vee b') + (a \wedge a')(b\wedge b') \geq ab + a'b'
$$
(for $a>a'$ and $b>b'$, equality holds, and for $a>a'$ and $b<b'$, the desired inequality can be recast as $(a-a')(b'-b)\geq 0$, which is indeed true; the remaining cases follow by symmetry). All the remaining terms in \eqref{eq:Hams} except for $V_{\pm}(\varphi_x - h_x)$ are common to both $H^{+}_{n,\lambda}$ and $H^{-}_{n,\lambda}$ and involve only \textit{a single} field variable $\varphi_y$ $y \in V_n$, each. Trivially, $f(a \vee b) + f(a \wedge b)= f(a)+ f(b)$ for all $f: \mathbb{R} \to \mathbb{R}$ and $a,b \in \mathbb{R}$, hence the presence of these terms is irrelevant for \eqref{eq:FKG_red2} to hold.
Finally, we need to check that
\begin{equation}\label{eq:FKG_checkV_+}
 V_{+}((\varphi_x \vee \varphi_x')- h_x) + V_{-}((\varphi_x \wedge \varphi_x')- h_x) \leq V_{+}(\varphi_x - h_x) + V_{-}( \varphi_x'- h_x) 
\end{equation}
holds for all $\varphi_x, \varphi_x', h_x \in \mathbb{R}$. But \eqref{eq:FKG_checkV_+} holds with equality whenever $\varphi_x \geq \varphi_x'$. On the other hand, if $\varphi_x < \varphi_x'$, then the left-hand side of \eqref{eq:FKG_checkV_+} reads $V_{+}(\varphi_x' - h_x) + V_{-}( \varphi_x- h_x)$, and the inequality \eqref{eq:FKG_checkV_+} follows because $V_{+}(\cdot - h_x)$ is decreasing and $V_{-}(\cdot - h_x)$ increasing, cf. \eqref{eq:pot}. This completes the proof of \eqref{eq:FKG_red2}, and thus of \eqref{FKG1}. Finally, the proof of \eqref{FKG1} for periodic boundary conditions is completely analogous, but somewhat simpler, since the first reduction to a finite-volume measure can be dispensed with (i.e. one works directly with $\mathcal{E}^{\text{tr}}_{U^c}(\cdot,\cdot)$, for $U \subset \mathbb{T}_L$, $L \geq 1$, in \eqref{eq:Hams}; no limit $n \to \infty$ is needed). This concludes the proof Lemma \ref{L:FKG1}.
\end{proof}

\begin{remark} \label{R:FKG_lattice_cond}  (FKG lattice condition)

\medskip
\noindent
The inequality \eqref{FKG1} has an important corollary. Let $L \geq 1$, $\Lambda_L = B(0,L)$ if $\mathcal{G}= \mathbb{Z}^d$ or $\Lambda_L = \mathcal{G} = \mathbb{T}_L$ in the case of periodic boundary conditions. Then \eqref{FKG1} implies in particular that $Q^{\mathbf{h}}_{\theta,L}$, the law of $(\xi^{\mathbf{h}}_x)_{x \in \Lambda_L}$ under $\mathbb{P}^{\mathcal{G}}_{\theta}$, has the following \textit{monotonicity property}: for any $x\in \Lambda_L$, the map $\Omega^{\Lambda_L \setminus \{ x\}} \to [0,1]$, $\omega \mapsto Q^{\mathbf{h}}_{\theta,L}(Y_x =1 \, | \, Y_z = \omega(z), \, z \in \Lambda_L \setminus \{ x\})$ is increasing. By a classical result, see for instance \cite{GG06}, Theorem 2.1, this is equivalent to saying that the measure $Q^{\mathbf{h}}_{\theta,L}$ satisfies the so-called \textit{FKG lattice condition}, i.e.
\begin{equation} \label{EQ:FKG_lattice_cond}
Q^{\mathbf{h}}_{\theta,L}(\omega\vee\omega')\cdot Q^{\mathbf{h}}_{\theta,L}(\omega\wedge\omega')\geq Q^{\mathbf{h}}_{\theta,L}(\omega)\cdot Q^{\mathbf{h}}_{\theta,L}(\omega'), \text{ for all } \omega, \omega' \in \Omega^{\Lambda_L}
\end{equation}
(given two configurations $\omega,\omega'\in\Omega^{K}$, $K \subset \mathcal{G}$, we define
$\omega\vee\omega'\in \Omega^{K} $ by $(\omega\vee\omega')(y)=\max\{\omega(y),\,\omega'(y)\}$, for $y\in K$, and similarly $\omega\wedge\omega'$ with $\max$ replaced
by $\min$). \hfill $\square$
\end{remark}

\section{Phase transition and a Margulis-Russo-type formula}\label{S:PT_RUSSO}

In this short section, we first collect certain properties concerning the critical parameters $h_{*}$ and $h_{**}$ mentioned in the Introduction, which are all essentially known (but seemingly not written anywhere) and hint at their proofs. We then give an explicit formula for the derivative of the probability of a generic (finite dimensional) event $A^h$ with respect to $h$. Without further ado, we begin with the following

\begin{thm}\label{T:h_**_properties} (Non-trivial strongly subcritical regime) $(d \geq 3, \, \theta \in (0,1))$

\medskip
\noindent The critical parameter $h_{**}(\theta)$ defined in \eqref{eq:h_**} satisfies 
\begin{align}
& h_{**} (\theta,d) < \infty, \text{ for all $d\geq 3$ and $\theta \in[0,1]$,}\label{NON} \\
& p_{\theta, L}(h) \leq c(\theta, h) e^{- c'(\theta,h) \cdot L^{\rho(\theta,h)}}, \text{ for some $\rho(\theta,h) \in (0,1]$, all $h> h_{**}$ and $L \geq 1$.} \label{NON1}
\end{align}
Moreover, there exists a constant $\gamma = \gamma (\theta,d) >0$ such that
\begin{equation}\label{eq:h_**bis}
h_{**}(\theta,d)=  \inf \big\{ h \in \mathbb{R} \, ; \,  \liminf_{L\to \infty}  p_{\theta, L}(h) < \gamma (\theta,d) \big\}.
\end{equation}
\end{thm}
\begin{proof}
These results follow by mimicking the proofs of Theorem 2.6 in \cite{RoS13} and Theorem 2.1 in \cite{PR13} (both deal with the more difficult massless case), with obvious changes. Following the line of argument of \cite{PR13}, one can actually choose $\rho=1$ in \eqref{NON1} whenever $d \geq 4$.
\end{proof}

Next, we derive a particular formula for the derivative of the probabilities of certain events with respect to the height parameter $h \in \mathbb{R}$, which will prove useful when attempting a comparison with the corresponding conditional influences in the next section. For $h \in \mathbb{R}$ and $A \subset \subset\{0,1\}^{K}$, $K \subset \mathcal{G}$ a measurable subset, recall the notation $A^h$ from \eqref{eq:events_levset} and the subsequent discussion.

\begin{proposition}
\label{P:Russo_formula} $(\theta\in [0,1],\, h\in\mathbb{R},\, K \subset \subset \mathcal{G}, \, A\subset \Omega^{K})$
 \begin{equation}\label{eq:Russo_formula}
-\frac{d \,  \mathbb{P}_{\theta}^{\mathcal{G}}[A^{h}]}{dh}= \mathbb{E}_{\theta}^{\mathcal{G}} \Big[1_{A^{h}}(\varphi)\cdot\Big(\sum_{x\in K} \kappa^{\text{tr}}_{x}\varphi_{x}\Big)\Big]
\end{equation}
 \end{proposition}
\begin{proof}
Using \eqref{eq:phi_K_law}, one can write
\[
\mathbb{P}_{\theta}^{\mathcal{G}}[A^{h}]=Z^{-1}\int_{\mathbb{R}^{K}}1_{A^{h}}(\varphi)\cdot\exp\Big[-\frac12 \mathcal{E}^{\text{tr}}_K(\varphi,\varphi)\Big]\prod_{x\in K}d\lambda(\varphi_{x}),
\]
where $Z= \int _{\mathbb{R}^{K}} \exp [-\frac12 \mathcal{E}^{\text{tr}}_K(\varphi,\varphi)]\prod_{x\in K}d\lambda(\varphi_{x})$. Substituting $\psi_{x}=\varphi_{x}-h$, for all $x\in K$, and expanding $\mathcal{E}^{\text{tr}}_K(\varphi,\varphi)=\mathcal{E}^{\text{tr}}_K(\psi+\mathbf{h},\psi+\mathbf{h})=\mathcal{E}^{\text{tr}}_K(\psi,\psi)+2\mathcal{E}^{\text{tr}}_K(\psi,\mathbf{h})+\mathcal{E}^{\text{tr}}_K(\mathbf{h},\mathbf{h})$, where $\mathbf{h}$ denotes the field indexed by $K$ with constant value $h$ everywhere, one obtains
\[
\mathbb{P}_{\theta}^{\mathcal{G}}[A^{h}]=Z^{-1}\int_{\mathbb{R}^{K}}1_{A^{0}}(\psi)\cdot \exp\Big[-\mathcal{E}^{\text{tr}}_K(\psi,\mathbf{h})- \frac12 \mathcal{E}^{\text{tr}}_K(\mathbf{h},\mathbf{h})\Big]\cdot \exp\Big[-\frac12 \mathcal{E}^{\text{tr}}_K(\psi,\psi)\Big]\prod_{x\in K}d\lambda(\psi_{x})
\]
(N.B.: one may view this as a kind of (elementary) Cameron-Martin formula, cf. \cite{KS91}, p.190), and thus
\begin{align*}
-\frac{d \, \mathbb{P}_{\theta}^{\mathcal{G}}[A^{h}]}{dh}=Z^{-1}\int_{\mathbb{R}^{K}} &1_{A^{0}}(\psi)\cdot\Big[\, \frac{d}{dh}\mathcal{E}^{\text{tr}}_K(\psi,\mathbf{h})+\frac12 \frac{d}{dh}\mathcal{E}^{\text{tr}}_K(\mathbf{h},\mathbf{h})\Big]\\
&\times \exp\Big[-\frac12 \mathcal{E}^{\text{tr}}_K(\psi+\mathbf{h},\psi+\mathbf{h})\Big]\prod_{x\in K}d\lambda(\psi_{x}).
\end{align*}
Recall that for $f,g: K \to \mathbb{R}$,
\begin{equation*}
\mathcal{E}^{\text{tr}}_K(f,g)= \frac{1}{2} \sum_{ x,y \in K} c^{\text{tr}}_{x,y} (f(y)-f(x))(g(y)-g(x))+ \sum_{x \in K} \kappa^{\text{tr}}_{x}f(x)g(x),
\end{equation*}
which is obtained from \eqref{eq:Trace_Dirichlet} by polarization. In particular, this yields $\mathcal{E}^{\text{tr}}_K(\psi,\mathbf{h})=h\sum_{x\in K}\kappa^{\text{tr}}_{x} \psi_{x}$
and similarly $\mathcal{E}^{\text{tr}}_K(\mathbf{h},\mathbf{h})=h^2 \sum_{x\in K}\kappa^{\text{tr}}_{x}$.
It follows that
\begin{align*}
-\frac{d \, \mathbb{P}_{\theta}^{\mathcal{G}}[A^{h}]}{dh}
&=Z^{-1}\int_{\mathbb{R}^{K}}1_{A^{0}}(\psi)\cdot\Big[\sum_{x\in K}\kappa^{\text{tr}}_{x}(\psi_{x}+h)\Big] \cdot\exp\Big[-\frac12 \mathcal{E}^{\text{tr}}_K(\psi+\mathbf{h},\psi+\mathbf{h})\Big]\prod_{x\in K}d\lambda(\psi_{x}) \\
 &=\mathbb{E}_{\theta}^{\mathcal{G}} \Big[1_{A^{h}}(\varphi)\cdot\Big(\sum_{x\in K} \kappa^{\text{tr}}_{x}\varphi_{x}\Big)\Big],
\end{align*}
which completes the proof.
\end{proof}

\begin{remark}\label{R:deriv_comments}
1) In case $A$ is an increasing event, the function $h \mapsto \mathbb{P}_{\theta}^{\mathcal{G}}[A^{h}]$ is decreasing, and the right-hand side of \eqref{eq:Russo_formula}, which can be rewritten as $ \text{Cov}^{\mathbb{P}_{\theta}^{\mathcal{G}}} \big(1_{A^{h}}(\varphi), \, \sum_{x\in K} \kappa^{\text{tr}}_{x}\varphi_{x}\big)$ is indeed non-negative by the FKG-inequality.

\medskip
\noindent 2) In the limiting regime $\theta = 1$, in which $\mathbb{P}_{\theta}^{\mathcal{G}}$ is simply a product measure, one recovers a classical differential formula of Margulis \cite{Ma74} and Russo \cite{Ru81}; see also \cite{Gr99}, Theorem (2.25). To see this, first observe that, when $\theta =1$, $\kappa^{\text{tr}}_{x} = 1$ for all $x \in K$, cf. \eqref{eq:trace_kill}. For increasing $A$, denote by $\text{Piv}_x(A)= \{\omega \in \Omega^K; \, 1_A(\omega^x) \neq 1_A(\omega_x) \}$ the event that $x$ is pivotal for $A$, and let $\text{Piv}_x(A^h)= \{ \xi^h \in \text{Piv}_x(A) \}$, cf. below \eqref{eq:events_crossing} for notation, which is measurable with respect to the $\sigma$-algebra generated by $\psi \equiv (\varphi_y)_{y\in K \setminus \{ x\}}$, and thus independent of $\varphi_x$. Note that the same holds for $A^h \backslash \text{Piv}_x(A^h)$ (indeed $A \cap (\text{Piv}_x(A))^c= \{ \omega;\, 1_A(\omega_x)=1 \}$). Moreover, since $A$ is increasing, $A \cap \text{Piv}_x(A) = \{ \omega(x)=1\} \cap \text{Piv}_x(A)$, for all $x \in K$. Hence,
\begin{align*}
 \mathbb{E}_{1}^{\mathcal{G}} [1_{A^{h}}(\varphi)\cdot \varphi_{x}]
&=  \mathbb{E}_{1}^{\mathcal{G}} [ \, \mathbb{E}_{1}^{\mathcal{G}}[ (1_{\text{Piv}_x(A^h)}+ 1_{\text{Piv}_x(A^h)^c})1_{A^{h}} \cdot \varphi_{x}| \psi]] \\[6pt]
&= \mathbb{E}_{1}^{\mathcal{G}} [ \, 1_{\text{Piv}_x(A^h)} \cdot \mathbb{E}_{1}^{\mathcal{G}}[ 1_{\varphi_x \geq h} \cdot \varphi_{x}| \psi] ] + \mathbb{E}_{1}^{\mathcal{G}} [ \, 1_{A^h \cap \text{Piv}_x(A^h)^c} \cdot \underbrace{\mathbb{E}_{1}^{\mathcal{G}}[ \varphi_{x} ]}_{=0}] \\[-6pt]
&= f(h) \cdot  \mathbb{P}_{1}^{\mathcal{G}} [ \text{Piv}_x(A^h)],
 \end{align*}  
for all $x \in K$, where $f(\cdot)= \mathbb{E}_{1}^{\mathcal{G}}[ 1_{\{ \varphi_x \geq \cdot\}} \cdot \varphi_{x}]$ denotes the standard Gaussian density. The presence of the factor $f(h)$ results from taking derivatives with respect to $h$ in \eqref{eq:Russo_formula} rather than the density $p = p(h)= \int_h^\infty f(t) dt$ (note that $ d p / dh = -f(h)$).

\medskip
\noindent 3) (The massless case). The term appearing in the right-hand side of \eqref{eq:Russo_formula} seems to exhibit a (drastically) different behavior for $\theta = 0$ compared to positive mass (recall our convention that $\mathcal{G}= \mathbb{Z}^d$ is tacitly understood whenever $\theta=0$). Indeed, when $\theta >0$, we have $\kappa^{\text{tr}}_{x} = P^x_{\theta}[\widetilde{H}_K = \infty] \geq P^x_{\theta}[X_1 = \Delta] = \theta$ uniformly in $x \in K$, whereas in the case $\theta =0$, one only picks up a ``surface term'' (i.e. the summation in \eqref{eq:Russo_formula} is effectively over $\partial_{\text{in}}K= \{ x \in K; \, \exists y \in \mathcal{G}\setminus K : \, y \sim x\}$).  \hfill $\square$

\end{remark}

\section{Dominating the conditional influences}\label{S:dom_inf}

In this section, we show how that, up to a small, uniform multiplicative constant, the right-hand side of \eqref{eq:Russo_formula} dominates
a sum of so-called conditional influences. 
Even though the applications we have in mind only require a comparison ``in $\ell^1$,'' we actually manage to match the corresponding terms pointwise, i.e. for fixed $x\in K$, with $K \subset\subset \mathcal{G}$, see Proposition \ref{P:dom_inf_ptwise} below. This is arguably much stronger. Once the desired comparison with influences is established, Proposition \ref{P:Russo_formula} can be paired with a corresponding influence theorem to yield a differential inequality for the probability of a generic increasing event. This is the object of Corollary \ref{C:diff_ineq}, which is the main result of this section.

We begin with a notion of influences, introduced by Graham and Grimmett in \cite{GG06}, that generalizes the  extensively studied case of product measures, and is suited to our purposes. For arbitrary $\theta\in [0,1]$, $h\in\mathbb{R}$, $K \subset \subset \mathcal{G}$, and increasing $A\subset \Omega^{K}$, we define the \textit{conditional influence} of $x \in K$ (on the event $A^h$) as
\begin{equation}\label{eq:def_infl}
\begin{split}
I^{\mathcal{G}}_{\theta} (A^h, x) &= \mathbb{P}_{\theta}^{\mathcal{G}}[A^{h}\, | \, \varphi_x \geq h] - \mathbb{P}_{\theta}^{\mathcal{G}}[A^{h}\, | \, \varphi_x < h] \\
&=\text{Var}(\xi_x^h)^{-1}\cdot \text{Cov}^{ \mathbb{P}_{\theta}^{\mathcal{G}}}(1_{A^h}, \xi_x^h)
\end{split}
\end{equation}
(recall \eqref{eq:def_conf} for notation). In particular, together with the FKG-inequality for $ \mathbb{P}_{\theta}^{\mathcal{G}}$, the second equality exhibits that this quantity is non-negative. 

The conditional influences satisfy the following inequalities, which follow from the Influence Theorem of \cite{GG06}, itself relying on the results of \cite{KKL88} and \cite{BKKKL92} in the independent case.
\begin{thm} \label{T:inf_thm} (Influence Theorem)

\medskip
\noindent There exists a (universal) constant $c_{\mathrm{inf}} >0$, such that, for all $\theta\in [0,1]$, $h\in\mathbb{R}$, $K \subset \subset \mathcal{G}$ and all increasing $A\subset \Omega^{K}$, letting $|| I^{\mathcal{G}}_{\theta} (A^h) ||_{K,\infty} = \max_{x \in K} I^{\mathcal{G}}_{\theta} (A^h, x)$, one has

\begin{equation}\label{eq:inf_thm_l_infinity}
|| I^{\mathcal{G}}_{\theta} (A^h) ||_{K,\infty} \geq c_{\mathrm{inf}} \cdot \mathrm{Var}^{\; \mathbb{P}_{\theta}^{\mathcal{G}}}(1_{A^{h}}) \frac{\log|K|}{|K|}.
\end{equation}
Moreover, with $|| I^{\mathcal{G}}_{\theta} (A^h) ||_{K,1} = \sum_{x\in K} I^{\mathcal{G}}_{\theta} (A^h, x)$,
\begin{equation} \label{eq:inf_thm_l_1}
|| I^{\mathcal{G}}_{\theta} (A^h) ||_{K,1} \geq c_{\mathrm{inf}} \cdot \mathrm{Var}^{\; \mathbb{P}_{\theta}^{\mathcal{G}}}(1_{A^{h}}) \log \Big(  \frac{1}{2 || I^{\mathcal{G}}_{\theta} (A^h) ||_{K,\infty}}\Big). 
\end{equation}
\end{thm}
\begin{proof}
The claim \eqref{eq:inf_thm_l_infinity} is a direct application of Theorem 2.2 in \cite{GG06} to the measure $Q^{h}_{\theta,K} =(\xi^{h}_x)_{x \in K} \circ \mathbb{P}^{\mathcal{G}}_{\theta}$ (on $\Omega^K$). Indeed, $Q^{h}_{\theta,K}$ is a positive measure that satisfies the FKG-lattice condition, see \eqref{EQ:FKG_lattice_cond}. Thus \eqref{eq:inf_thm_l_infinity} follows immediately from Eqn. (2.9) in \cite{GG06} (note that the factor $\min \{ \mu(A), 1- \mu(A)\}$ appearing there can be replaced by $\mu(A)(1- \mu(A))= \text{Var}^{\mu}(1_A)$ using the elementary inequality $x \wedge (1-x) \geq x(1-x)$, valid for all $x \in [0,1]$).

Although the ``$\ell^1$-estimate'' corresponding to \eqref{eq:inf_thm_l_1}, which, in the notation of Theorem 2.2 of \cite{GG06}, would read
\begin{equation} \label{eq:inf_thm_l_1_gen}
\sum_{i=1}^N I_A(i) \geq c \mu(A)(1- \mu(A)) \log \big(1/2 \max_{1\leq i \leq N} I_A(i) \big), \text{ for all increasing } A \subset \{ 0,1\}^N
\end{equation}
(here $\mu$ denotes a generic positive measure on $ \{ 0,1\}^N$ satisfying the lattice-FKG condition) does not appear in \cite{GG06}, this result still holds under the assumptions of Theorem 2.2 in \cite{GG06}, and follows by inspection of its proof. We briefly explain this. Adopting the notation of \cite{GG06}, and given $A$ as in \eqref{eq:inf_thm_l_1_gen}, one constructs in the proof of Theorem 2.2 in \cite{GG06} an increasing subset $B$ of $[0,1]^N$ (endowed with the $N$-dimensional Lebesgue measure $\lambda^N$) with the following two properties: P1) $\lambda^N(B) = \mu(A)$, and P2) $I_A(i) \geq J_B(i)$ for all $1\leq i \leq N$, where $J_B(i)= \lambda^{N-1}(u;\, t \in [0,1] \mapsto 1_B(u_1,\dots u_{i-1} t, u_i,\dots u_N) \text{ is not constant})$ is the influence considered in \cite{BKKKL92}. In particular,
$$
\sum_{i=1}^N I_A(i) \stackrel{\mathrm{P2)}}{\geq} \sum_{i=1}^N J_B(i) \geq c \lambda^N(B) (1-\lambda^N(B)) \log \big(1/2 \max_{1\leq i \leq N} J_B(i) \big)
$$
where the last step is due to the $\ell^1$-result in \cite{BKKKL92}; see the Remark therein at the very end. One then obtains \eqref{eq:inf_thm_l_1_gen} using P1) to replace the variance and the fact that $x \mapsto \log(1/2x)$ is decreasing for $x > 0$, along with P2). Finally, the claim \eqref{eq:inf_thm_l_infinity} follows by applying \eqref{eq:inf_thm_l_1_gen} to $Q^{h}_{\theta,K} $, as above.
\end{proof}

The key to a successful application of Theorem \ref{T:inf_thm} lies in the comparison between the formula \eqref{eq:Russo_formula} for the derivative $ - d \,  \mathbb{P}_{\theta}^{\mathcal{G}}[A^{h}] /dh$, when $A$ is increasing, with the corresponding ``$\ell^1$-norm'' of influences for the event $A^h$. This comes in the next
\begin{proposition} \label{P:dom_inf_ptwise} $(\theta \in (0,1), \, M > 0)$

\medskip
\noindent There exists a constant $c_1(\theta, M)>0$ such that, for all $K \subset \subset \mathcal{G}$, all increasing sets $A \subset \Omega^K$, all $x \in K$ and $ h \in (- M, M)$, 
\begin{equation}\label{EQ:dom_inf_00}
\mathbb{E}_{\theta}^{\mathcal{G}} [1_{A^{h}}(\varphi)\cdot \kappa^{\text{tr}}_{x}\varphi_{x}] \geq c_1(\theta, M) \cdot I^{\mathcal{G}}_{\theta} (A^h, x),
\end{equation}
and 
\begin{equation}\label{EQ:Russo_formula_inf}
-\frac{d \,  \mathbb{P}_{\theta}^{\mathcal{G}}[A^{h}]}{dh} \geq c_1(\theta, M) \cdot || I^{\mathcal{G}}_{\theta} (A^h) ||_{K,1}.
\end{equation}
\end{proposition}
%

We emphasize the following aspect of Proposition \ref{P:dom_inf_ptwise}. It is evidently crucial for the constant $c_1$ appearing in \eqref{EQ:dom_inf_00} and \eqref{EQ:Russo_formula_inf} to have no (or little) dependence on $K$ (or $L$, in the case of periodic boundary conditions). Before proving Proposition  \ref{P:dom_inf_ptwise}, we collect the following result, which will play a key role in proving our main Theorem \ref{T:MAIN} in the next section, but is also interesting in its own right. Recall that an event $\mathcal{A} \in \Omega^{\mathbb{T}_L}$ is called \textit{translation invariant} whenever $\mathcal{A}= \tau_x(\mathcal{A}) \equiv \{ \tau_x \omega ;  \; \omega \in \mathcal{A} \}$ for all $x \in \mathbb{T}_L$, where $\tau_x \omega \in \Omega^{\mathbb{T}_L}$ is the configuration with $(\tau_x \omega ) (y) = \omega(x+y)$, for $y \in \mathbb{T}_L$.

\begin{corollary}\label{C:diff_ineq} (Differential inequalities) $(\theta \in (0,1), \, M > 0)$

\medskip
\noindent There exists $c_2(\theta, M)>0$ such that, for all $K \subset \subset \mathcal{G}$, all increasing sets $A \subset \Omega^K$, and all $h \in (- M, M)$,
\begin{equation}\label{EQ:diff_ineq_gen}
-\frac{d \,  \mathbb{P}_{\theta}^{\mathcal{G}}[A^{h}]} {dh} \geq c_2(\theta, M) \cdot \mathrm{Var}^{\; \mathbb{P}_{\theta}^{\mathcal{G}}}(1_{A^{h}}) \log \Big(  \frac{1}{2 || I^{\mathcal{G}}_{\theta} (A^h) ||_{K,\infty}}\Big).
\end{equation}
Moreover, for all $L \geq 1$,  $h \in (-M, M)$, and all translation invariant events $\mathcal{A} \in \Omega^{\mathbb{T}_L}$ (with $\mathbb{P}_{\theta}^{L}\equiv  \mathbb{P}_{\theta}^{\mathbb{T}_L} $),
\begin{equation}\label{EQ:diff_ineq_trans_inv}
-\frac{d \, \mathbb{P}_{\theta}^{ L}[\mathcal{A}^{h}]} {dh} \geq c_2(\theta, M) \cdot \mathrm{Var}^{\, \mathbb{P}_{\theta}^{L}}(1_{\mathcal{A}^{h}}) \log |\mathbb{T}_L|.
\end{equation}
\end{corollary}
\begin{proof}
The inequality \eqref{EQ:diff_ineq_gen} follows immediately from \eqref{EQ:Russo_formula_inf} and the $\ell^1$-influence theorem \eqref{eq:inf_thm_l_1}, with $c_2(\theta, M) = c_{\mathrm{inf}} \cdot c_1(\theta, M)$. As for \eqref{EQ:diff_ineq_trans_inv}, notice that, if $\mathcal{A}$ is translation invariant, all influences must coincide, i.e. $I_{\theta}^{ L} (\mathcal{A}^h, x)= I_{\theta}^{ L} (\mathcal{A}^h, y)$ for all $x, y \in \mathbb{T}_L$ (with hopefully obvious notation, $I_{\theta}^{L}(\cdot)$ refers to an influence with respect to the measure $\mathbb{P}_{\theta}^{L}$). This follows immediately upon writing $I_{\theta}^{L} (\mathcal{A}^h, x) = c(\theta,L,h)\cdot \text{Cov}^{ \mathbb{P}_{\theta}^{L}}(1_{\mathcal{A}^h}, \xi_x^h)$, for a suitable constant $c(\theta,L,h)$ and all $x\in \mathbb{T}_L$, cf. \eqref{eq:def_infl}, and using translation invariance of $\mathbb{P}_{\theta}^{ L}$ to deduce that the covariance does not depend on $x$. Hence, by the $\ell^{\infty}$-bound of Theorem \ref{T:inf_thm},
$$
|| I_{\theta}^{ L} (\mathcal{A}^h) ||_{\mathbb{T}_L,1} = \sum_{x \in \mathbb{T}_L}I_{\theta}^{ L} (\mathcal{A}^h, x) = |\mathbb{T}_L|\cdot || I_{\theta}^{L} (\mathcal{A}^h) ||_{\mathbb{T}_L,\infty} \stackrel{\eqref{eq:inf_thm_l_infinity}}{\geq} c_{\text{inf}} \cdot \mathrm{Var}^{\; \mathbb{P}_{\theta}^{L}}(1_{\mathcal{A}^{h}}) \log |\mathbb{T}_L|,
$$
which, together with \eqref{EQ:Russo_formula_inf}, yields \eqref{EQ:diff_ineq_trans_inv}.
\end{proof}

We now proceed to the proof of Proposition \ref{P:dom_inf_ptwise}. The proof exhibits that the objects on either side of \eqref{EQ:dom_inf_00} have a very similar structure, and essentially comprises a Riemann sum argument (with carefully chosen mesh sizes) to compare the two integrals.

\medskip
\noindent{\textit{Proof of Proposition \ref{P:dom_inf_ptwise}.}} Let $\theta \in (0,1)$ and $M >0$ be fixed. We may assume without loss of generality that $M \geq 1$. We will show \eqref{EQ:dom_inf_00}. The assertion then immediately follows from \eqref{EQ:dom_inf_00} upon summing over $x \in K$, using the differential formula \eqref{eq:Russo_formula}. 

We first consider the influence $I^{\mathcal{G}}_{\theta} (A^h, x)$ appearing on the right-hand side of \eqref{EQ:dom_inf_00}. For arbitrary $K \subset \subset \mathcal{G}$, $x \in K $, $A \subset \Omega^K$ increasing and $ h \in(-M, M) $, we introduce the function 
\begin{equation} \label{eq:def_mu}
\mu_{A,x}^{h}: \mathbb{R}\to [0,1],  \qquad t \mapsto \widetilde{\mathbb{P}}^{\mathcal{G}}_{\theta} [A^h(\varphi) | \, \varphi_x = t] \stackrel{\eqref{eq:cond_0}}{=}{\mathbb{P}}^{\mathcal{G}}_{\theta,x} [A^h({\varphi}_{\cdot}+ p_{\cdot x}t)] 
\end{equation} 
(for clarity, its dependence on $\mathcal{G}$ and $\theta$ is kept implicit) where ${\mathbb{P}}^{\mathcal{G}}_{\theta,x} $ refers to the field with covariance $g_{\theta, U}(\cdot,\cdot)$ with $U= \{ x\}$, see \eqref{eq:GFgen}, corresponding to a random walk killed uniformly at rate $\theta$ or when hitting $x$, and $p_{yx}=P^y_{\theta}[H_{\{ x\}} < \infty]$, for $y \in \mathcal{G}$, cf. Lemma \ref{L:cond_exps}. Note that, by monotonicity of ${A}^h$,
\begin{equation}\label{eq:b_monot}
\text{${\mu}_{A,x}^{h}(\, \cdot \,)$ is an increasing function.}
\end{equation}
By definition, see \eqref{eq:def_infl}, and by virtue of \eqref{EQ:st_phi_cond_exps}, the conditional influence can be rewritten, upon conditioning on $\varphi_x$,  as
\begin{equation*}
 I^{\mathcal{G}}_{\theta} ({A}^h, x)= E^{U}[{\mu}_{A,x}^{h}(U)]- E^{V}[{\mu}_{A,x}^{h}(V)],
\end{equation*}
where $U \stackrel{\mathrm{law}}{=} \varphi_x \, | \, \varphi_x \geq h$ under $P^U$, $V \stackrel{\mathrm{law}}{=} \varphi_x \, | \, \varphi_x \leq h$ under $P^V$ (of course, both depend on $\theta$ and $h$), and $U$, $V$ and $\varphi$ are independent. In particular $U \geq V$. It is then plain, using \eqref{eq:b_monot},  that $ I^{\mathcal{G}}_{\theta} ({A}^h, x) \geq 0$ (although this also follows from the FKG-inequality, as explained above). Letting $W= \max\{ |U|, |V|\}$, and on account \eqref{eq:b_monot}, we can bound
\begin{equation}\label{eq:I_bound_W}
 I^{\mathcal{G}}_{\theta} ({A}^h, x) \leq E^{W}[{\mu}_{A,x}^{h}(W)- {\mu}_{A,x}^{h}(-W)].
\end{equation}
The distribution of $W$ has Gaussian tails. More precisely, for $ h \in [0, M]$ and $t\geq 0$,
\begin{equation*}
\begin{split}
P[ W > t] &\leq P^{U}[ U > t] + P^{V}[ |V| > t] \\
&= \frac{ \mathbb{P}^{\mathcal{G}}_{\theta}[\varphi_x > t \vee h]}{\mathbb{P}^{\mathcal{G}}_{\theta}[\varphi_x >  h]} +\frac{\mathbb{P}^{\mathcal{G}}_{\theta}[\varphi_x <-t] + \mathbb{P}^{\mathcal{G}}_{\theta}[t < \varphi_x < h]1\{ t \leq h \}}{\mathbb{P}^{\mathcal{G}}_{\theta}[\varphi_x \leq h]} \\[0.4em]
&\leq C(\theta, M) \cdot  \mathbb{P}^{\mathcal{G}}_{\theta}[\varphi_0 > t ]
\end{split}
\end{equation*}
for some (large) constant $C(\theta, M)$ (the last line is clearly true for $t\geq h$, and follows for $t \leq h$ by adapting the constant $C(\theta, M)$, minding that $h \leq M$). Moreover, this bound continues to hold for $ h \in [-M,0]$ by symmetry, since $W \equiv W(h) \stackrel{\text{law}}{=}W(-h)$. In particular, for any increasing sequence $(t_k)_{k\geq 0}$, with $t_0 =0$ and $\lim_{k} t_k = \infty$, \eqref{eq:I_bound_W} implies that
\begin{equation}\label{eq:UB_inf2}
\begin{split}
I^{\mathcal{G}}_{\theta} ({A}^h, x)  &\leq \sum_{k\geq 0} P^{W}[t_k \leq W < t_{k+1}]\cdot ({\mu}_{A,x}^{h}(t_{k+1})- {\mu}_{A,x}^{h}(-t_{k+1})) \\
&\leq C(\theta, M) \sum_{k\geq 1} \mathbb{P}^{\mathcal{G}}_{\theta}[\varphi_0 > t_{k-1}]\cdot ({\mu}_{A,x}^{h}(t_{k})- {\mu}_{A,x}^{h}(-t_{k})),
\end{split}
\end{equation}
for any $h \in [-M,M]$. On the other hand, letting $\Lambda_k = (-t_{k+1}, t_{k+1}) \backslash (-t_k,t_k)$, the derivative term appearing on the left-hand side of \eqref{EQ:dom_inf_00}, which is non-negative by the FKG-inequality, can be rewritten as
\begin{equation}\label{eq:LB_deriv2}
\begin{split}
\mathbb{E}_{\theta}^{\mathcal{G}}[1_{{A}^h}(\varphi) \cdot  \kappa^{\mathrm{tr}}_{x} \varphi_x]
&=  \kappa^{\mathrm{tr}}_{x} \cdot \mathbb{E}_{\theta}^{\mathcal{G}}[{\mu}_{A,x}^{h}(\varphi_x) \cdot \varphi_x] \\
&=  \kappa^{\mathrm{tr}}_{x} \cdot \sum_{k\geq 0}\mathbb{E}_{\theta}^{\mathcal{G}}[{\mu}_{A,x}^{h}(\varphi_x) \cdot \varphi_x 1\{ \varphi_x \in \Lambda_k\}] \\
& \geq \theta \cdot  \sum_{k\geq 1}\mathbb{E}_{\theta}^{\mathcal{G}}[  \varphi_x 1\{ t_k \leq \varphi_x < t_{k+1}\}]\cdot({\mu}_{A,x}^{h}(t_k) - {\mu}_{A,x}^{h}(-t_k)),
\end{split}
\end{equation}
where the last line follows by monotonicity of ${\mu}_{A,x}^{h}(\, \cdot \,)$ and the symmetry $\varphi_x$, and we used $\kappa^{\text{tr}}_{x} = P_{\theta}^{x}[\widetilde{H}_K = \infty] \geq P_{\theta}^{x}[X_1 = \Delta]$, cf. \eqref{eq:trace_cond}. Define now $t_1=1$ and $t_{k+1}= t_k + \frac{1}{t_k}$, for all $k \geq 1$ (observe that this sequence is strictly increasing and indeed satisfies $\lim_{k}t_k =\infty$). By the mean value theorem, we can bound, for any $k \geq 1$, and some $\tau_k \in [t_k,t_{k+1}]$, with $f_{\theta}(\cdot)$ denoting the density of $\varphi_x$,
\begin{equation*}
\begin{split}
\mathbb{E}_{\theta}^{\mathcal{G}}[  \varphi_x 1\{ t_k \leq \varphi_x < t_{k+1}\}] &\geq \mathbb{P}_{\theta}^{\mathcal{G}}[ t_k \leq \varphi_x < t_{k+1}] \\
&= (t_{k+1} - t_k) f_{\theta}(\tau_k) \\
& = \frac{1}{t_k} f_{\theta}(t_{k-1})\cdot \text{exp}\Big[-\frac{\tau_k^2 - t_{k-1}^2}{2g_{\theta}(0)}\Big]
\end{split}
\end{equation*}
(N.B. the first estimate might seem crude, but $ \varphi_x$ is a standard Gaussian and ``sizeable'' for any $k \geq 1$). By definition of the sequence $(t_k)_{k\geq 0}$, we have, for any $k \geq 2$,
\begin{equation*}
\begin{split}
\tau_k^2 - t_{k-1}^2 &\leq (t_{k+1}+ t_{k-1})(t_{k+1}- t_{k-1})\\
& \leq (t_{k+1}+ t_{k})(t_{k-1}^{-1}+ t_{k}^{-1})=\Big( 1+ \frac{t_{k+1}}{t_k} \Big)\cdot  \Big( 1+ \frac{t_{k}}{t_{k-1}} \Big).
\end{split}
\end{equation*}  
Observe that, by definition, for any $k \geq 1$,  $\frac{t_{k+1}}{t_k}=1+\frac{1}{t_k^2} \to 1$ as $k\to \infty$ and hence $ \limsup_{k\to \infty} (\tau_k^2 - t_{k-1}^2) \leq 4$. In particular, the sequence $(\tau_k^2 - t_{k-1}^2)_{k \geq 1}$ is bounded. Substituting in the above estimate yields, for all $k \geq 2$,
\begin{equation*}
\begin{split}
\mathbb{E}_{\theta}^{\mathcal{G}}[  \varphi_x 1\{ t_k \leq \varphi_x < t_{k+1} \}] &\geq \frac{t_{k-1}}{t_k} \cdot \frac{f_{\theta}(t_{k-1})}{t_{k-1}} e^{-C/2g_{\theta}(0)} \\
&\geq c(\theta) e^{-C/2g_{\theta}(0)} \cdot   \mathbb{P}_{\theta}^{\mathcal{G}}[\varphi_x > t_{k-1}],
\end{split}
\end{equation*}
where we used the fact that $(\frac{t_{k-1}}{t_k})_{k\geq 2}$ is a strictly positive sequence converging to $1$, and therefore uniformly bounded from below by a positive constant, and an elementary Gaussian tail estimate in the last step. By adapting the constant $c$, one can ensure that the last estimate holds for $k = 1$ as well, and substituting in \eqref{eq:LB_deriv2} yields
\begin{equation*}
\begin{split}
\mathbb{E}_{\theta}^{\mathcal{G}}[1_{{A}^h}(\varphi) \cdot  \kappa^{\mathrm{tr}}_{x} \varphi_x] &\geq c(\theta) \sum_{k\geq 1}  \mathbb{P}_{\theta}^{\mathcal{G}}[\varphi_x > t_{k-1}]\cdot({\mu}_{A,x}^{h}(t_k) - {\mu}_{A,x}^{h}(-t_k))\\ 
&\hspace{-1ex}\stackrel{\eqref{eq:UB_inf2}}{\geq}c(\theta) \cdot C(\theta,M)^{-1} I^{\mathcal{G}}_{\theta}({A}^h,x).
\end{split}
\end{equation*}
This completes the proof of \eqref{EQ:dom_inf_00}.  \hfill $\square$

\begin{remark}\label{R:inf_theta=0} 
Even though the influence theorem \ref{T:inf_thm} continues to hold for $\theta=0$, one cannot expect the inequality \eqref{EQ:Russo_formula_inf} to hold as such in the massless case. This is due to the (much) stronger correlations. Indeed, consider for instance the event $A^{h}= \{ \varphi_0 \geq h\}$, with, say, $h=0$.
On the one hand, for all $x \in \mathbb{Z}^d$, with $U, V$ as above, cf. below \eqref{eq:b_monot}, 
\begin{equation*}
\begin{array}{rcl}
I^{\mathbb{Z}^d}_{\theta=0} (A^0, x) \hspace{-1ex} & \stackrel{\eqref{eq:def_infl}}{=}  & \hspace{-1ex} \mathbb{P}^{\mathbb{Z}^d}_{0,x} \otimes P^U[\varphi_0 + p_{0x}U \geq 0] -  \mathbb{P}^{\mathbb{Z}^d}_{0,x} \otimes P^V [\varphi_0 + p_{0x}V \geq 0] \\[0.4em]
\hspace{-1ex} & \geq & \hspace{-1ex} c \cdot  \mathbb{P}^{\mathbb{Z}^d}_{0,x} \otimes P^U \otimes P^V [ - p_{0x} \leq \varphi_0 \leq  p_{0x}, \, U \wedge |V| \geq 1] \\
\hspace{-1ex} & \geq & \hspace{-1ex} c' \cdot p_{0x} = c' \cdot P_0^0[H_{\{ x \}} < \infty] \stackrel{\eqref{EQ:st_entr_prob}}{\geq} c'' g_0(x) \geq c \cdot |x|^{-(d-2)},
\end{array}
\end{equation*}
(see for instance \cite{La91}, p.31, Thm. 1.5.4 regarding the last estimate) and therefore 
$$
\sum_{x \in S_L} I^{\mathbb{Z}^d}_{\theta=0} (A^0, x) = \Theta(L) \rightarrow \infty \text{ as } L \to \infty.
$$
On the other hand, $-\frac{d \,  \mathbb{P}_{0}^{\mathbb{Z}^d}[A^{h}]} {dh}\big|_{h=0} = f_0(0)= (2\pi g_0(0))^{-1/2}$. \hfill $\square$
\end{remark}

\section{Crossing probabilities near $h_{**}$} \label{S:APPS}

In this section, we apply the differential inequalities obtained in Corollary \ref{C:diff_ineq} to investigate crossing probabilities near the critical parameter $h_{**}$. To begin with, we prove a weaker version of our main Theorem (without speed of convergence), see Proposition \ref{P:NO_SPEED} below. We begin with this argument because it is conceptually important, as it displays explicitly that the influences for the crossing events of interest are small (as needed for the successful deployment of \eqref{EQ:diff_ineq_gen}). Moreover, one can work directly on $\mathbb{Z}^d$. The proof of Theorem \ref{T:MAIN} is presented thereafter, and involves working with a suitably chosen translation invariant event (on the torus), for which the polynomial speed of convergence essentially comes for free, cf. \eqref{EQ:diff_ineq_trans_inv}. While the argument is arguably slick, the geometric insight that the corresponding influences are small is rather implicit. We remind the Reader of the definition of $p_{\theta,L}(\cdot)$ in \eqref{eq:def_p}, and further adopt the following notation. While $\mathbb{P}_{\theta,U}^{\mathcal{G}}$ was very convenient for treating all types of boundary conditions on the same footing, we will now have to change back and forth between them, and agree that henceforth,
\begin{equation}\label{EQ:notation}
\mathbb{P}_{\theta,U} \equiv \mathbb{P}_{\theta,U}^{\mathbb{Z}^d}, \text{ for $U\subset \mathbb{Z}^d$}, \qquad  \mathbb{P}_{\theta,U}^{ \,  L} \equiv \mathbb{P}_{\theta,U}^{\mathbb{T}_L}, \text{ for $U\subset \mathbb{T}_L$ and $L \geq 1$}
\end{equation}
in order to avoid occasional confusion. We omit $U$ from the notation in \eqref{EQ:notation} whenever $U=\emptyset$, and simply write 
$\mathbb{P}_{\theta,x} $, resp. $\mathbb{P}_{\theta,x}^{\, L} $, if $U= \{ x\}$ for some $x \in \mathbb{Z}^d$.

\begin{proposition}\label{P:NO_SPEED}$(\theta \in (0,1))$
\begin{equation}\label{EQ:MAINbis}
\lim_{L \to \infty} p_{\theta, L}(h) = 1, \text{ for all $h< h_{**}(\theta)$}. 
\end{equation}
\end{proposition}

\begin{proof} Fix $\theta \in (0,1)$. If $h_*(\theta)= h_{**}(\theta)$, then by ergodicity of $\mathbb{P}_{\theta}$, one has, for all $h< h_{**}(\theta)$,
\begin{equation*}
\begin{split}
1 &= \mathbb{P}_{\theta}[E^{\geq h} \text{ contains an infinite cluster}] 
=  \mathbb{P}_{\theta}\Big[\bigcup_{L \geq 1} \{ S_L \stackrel{\geqslant h}{\longleftrightarrow} \infty \} \Big] \\
& = \lim_{L\to \infty} \uparrow \mathbb{P}_{\theta}[ S_L \stackrel{\geqslant h}{\longleftrightarrow} \infty \}] \leq \lim_{L\to \infty} \mathbb{P}_{\theta}[ B_L \stackrel{\geqslant h}{\longleftrightarrow} S_{2L} \}] 
\end{split}
\end{equation*}
(for the last inequality, one considers $\limsup_L$ and $\liminf_L$ separately), and \eqref{EQ:MAINbis} follows. From now on, suppose $h_*(\theta) < h_{**}(\theta)$. Moreover, select $M = M(\theta) = 1+ |h_{**}(\theta)| $, so that $h_{**}(\theta) \in (-M,M)$, where $M$ refers to the parameter appearing in Proposition \ref{P:dom_inf_ptwise}. By monotonicity of $p_{\theta, L}(\cdot)$, it suffices to show \eqref{EQ:MAINbis} with $h= h_{**}- 2\delta$, for all
\begin{equation}\label{eq:delta_interval}
0< \delta <\frac{1 \wedge (h_{**}- h_*)}{4}\equiv \bar{\delta}(\theta). 
\end{equation}
In particular, note that for all $\delta$ satisfying \eqref{eq:delta_interval}, one has $  h_{**}- 2\delta > h_*$, and therefore
\begin{equation}\label{eq:delta_subcrit}
\mathbb{P}_{\theta}[0 \stackrel{\geqslant h_{**}- 2\delta}{\longleftrightarrow} \infty]=0, \text{ for all $\delta \in (0, \bar{\delta}(\theta))$}.
\end{equation}
From now on, fix an arbitrary $\delta$ in the interval $(0, \bar{\delta}(\theta))$, and abbreviate $\mathscr{C}_L^h =\{ B_L \stackrel{\geqslant h}{\longleftrightarrow}  S_{2L}\}$, for $h \in \mathbb{R}$ and $L \geq 1$. For clarity, we will henceforth keep the dependence of $h_{**}$, $\bar{\delta}$ and $M$  on $\theta$ implicit. The argument is indirect. Thus, contrary to \eqref{EQ:MAINbis}, we assume that
\begin{equation} \label{eq:contr_ass}
\text{there exists $\eta > 0$ such that } \limsup_{L\to \infty} \mathbb{P}_{\theta}[(\mathscr{C}_L^{h_{**}-2\delta})^c] \geq \eta,
\end{equation}
and we will show that this leads to a contradiction. In accordance with \eqref{eq:contr_ass}, let $(L_n)_{n \geq 0}$, where $L_n = L_n(\delta, \theta)$, be a growing sequence of length scales with $\lim_{n\to \infty}L_n = \infty$ and such that
\begin{equation} \label{eq:contr_ass_bis}
 \mathbb{P}_{\theta}[(\mathscr{C}_{L_n}^{h_{**}-2\delta})^c] \geq \eta, \text{ for all $n\geq 0$}.
\end{equation}
By definition of $h_{**}$ in Theorem \ref{T:h_**_properties}, see \eqref{eq:h_**bis}, we can ensure (possibly redefining $(L_n)_{n \geq 0}$ by neglecting its first few terms, in a manner depending on $\delta$) that
\begin{equation}\label{eq:cross_LB1}
 \mathbb{P}_{\theta}[\mathscr{C}_{L_n}^{h_{**}-\delta}] \geq \gamma(\theta)/2, \text{ for all $n\geq 0$}.
\end{equation}
The differential inequality \eqref{EQ:diff_ineq_gen} (with $\mathcal{G}=\mathbb{Z}^d$ and our above choice of $M$) applied to the (increasing) crossing events $\mathscr{C}_L^h $ along the sequence $(L_n)_n$ yields
\begin{equation}\label{eq:MAINbis_diff_ineq}
\begin{split}
&-\frac{d \,  \mathbb{P}_{\theta}[\mathscr{C}_{L_n}^{h}]} {dh} \geq c(\theta) \cdot \mathrm{Var}^{\; \mathbb{P}_{\theta}}(1_{\mathscr{C}_{L_n}^{h}}) \log \Big(  \frac{1}{2 || I_{\theta} (\mathscr{C}_{L_n}^{h}) ||_{B_{2L_n}, \infty}}\Big), \\[0.4em]
&\text{ for all $n \geq 0$ and $h= h_{**}-\ell$ with $\ell \in [\delta, 2\delta]$}
\end{split}
\end{equation}
(observe that $[h_{**} -2 \delta, h_{**} - \delta] \subset (- M, M)$ for all $\delta$ satisfying \eqref{eq:delta_interval}). The variance on the right-hand side of \eqref{eq:MAINbis_diff_ineq} can bounded from below, for all $h$ of the given form and $n\geq 0$, by
\begin{equation*}
 \mathrm{Var}^{\; \mathbb{P}_{\theta}}(1_{\mathscr{C}_{L_n}^{h}}) =  \mathbb{P}_{\theta}[\mathscr{C}_{L_n}^{h}]\cdot (1- \mathbb{P}_{\theta}[\mathscr{C}_{L_n}^{h}]) \geq  \mathbb{P}_{\theta}[\mathscr{C}_{L_n}^{h_{**}-\delta}]\cdot  \mathbb{P}_{\theta}[(\mathscr{C}_{L_n}^{h_{**}-2\delta})^c] \stackrel{\eqref{eq:contr_ass_bis}, \eqref{eq:cross_LB1}}{\geq} c'(\theta) \ (>0). 
 \end{equation*}
The usefulness of \eqref{eq:MAINbis_diff_ineq} hinges crucially on an upper bound for the $\ell^{\infty}$-norm of the influences for the crossing events $\mathscr{C}_{L}^{h}$, which comes in the next

\begin{lem} \label{L:inf_bound_BIS}$(\theta \in (0,1), \, \delta \in (0,\bar{\delta}), \, \text{\eqref{eq:delta_subcrit}})$
\begin{equation}\label{eq:inf_bound_BIS}
\lim_{L\to \infty} \sup_{\ell \in [\delta, 2\delta]} || I_{\theta} (\mathscr{C}_{L}^{h_{**}-\ell}) ||_{B_{2L},\infty} = 0.
\end{equation}
\end{lem}

Suppose for a moment that Lemma \ref{L:inf_bound_BIS} holds. We first explain how to reach the desired contradiction, and complete the proof of \eqref{EQ:MAINbis}. By \eqref{eq:MAINbis_diff_ineq}, we have, for all $n \geq 0$ and $h= h_{**}-\ell$ with $\ell \in [\delta, 2\delta]$, using the above lower bound on the variance, and since $\log(1/2x)$ is decreasing for $x>0$,
$$
-\frac{d \,  \mathbb{P}_{\theta}[\mathscr{C}_{L_n}^{h}]} {dh} \geq c''(\theta) \cdot  \log \Big(  \frac{1}{2  \sup_{\ell \in [\delta, 2\delta]} || I_{\theta} (\mathscr{C}_{L_n}^{h_{**}-\ell}) ||_{B_{2L_n},\infty }}\Big).
$$
Note that the right-hand side of this estimate does not involve the (integration) variable $h$. In particular, because $\log(1/2x) \to \infty$ as $x \to 0$, and by virtue of \eqref{eq:inf_bound_BIS}, we may ensure, choosing $\tilde{n}= \tilde{n}(\theta, \delta)$ sufficiently large, that $-d \,  \mathbb{P}_{\theta}[\mathscr{C}_{L_{\tilde{n}}}^{h}] /dh \geq 2\delta^{-1}$, thus yielding, upon integrating over the interval $[h_{**}-2\delta, \, h_{**}-\delta]$,
$$
\mathbb{P}_{\theta}[\mathscr{C}_{L_{\tilde{n}}}^{h_{**}-2\delta}] - \mathbb{P}_{\theta}[\mathscr{C}_{L_{\tilde{n}}}^{h_{**}-\delta}] \geq 2,
$$
a contradiction. Consequently, the assumption \eqref{eq:contr_ass} is false, i.e. \eqref{EQ:MAINbis} holds. It remains to prove the lemma.

\bigskip

\noindent \textit{Proof of Lemma \ref{L:inf_bound_BIS}}. For clarity, we keep the dependence of parameters on $\theta$ implicit throughout the proof. Let  $\varepsilon > 0$, $x \in \mathbb{Z}^d$ and $h = h_{**}-\ell$ for some $\ell \in [\delta, 2\delta]$. We also introduce a parameter $R \geq 1$ to be chosen later. The conditional influence of $x$ on the occurrence of $\mathscr{C}_{L}^{h}$ can be rewritten, by conditioning on $\varphi_x$ (see below \eqref{eq:b_monot} for notation) as
\begin{equation*}
I_{\theta}(\mathscr{C}_{L}^{h},x) = \mathbb{P}_{\theta,x}\otimes P^U[\mathscr{C}_{L}^{h}(\varphi_{\cdot}+p_{\cdot x}U)] -  \mathbb{P}_{\theta,x}\otimes P^V[\mathscr{C}_{L}^{h}(\varphi_{\cdot}+p_{\cdot x}V)]. 
\end{equation*}
The variables $U$ and $V$ depends on $h$, which varies over a compact set, and has Gaussian tails. Hence, one can find $u = u(\delta, \varepsilon) >0$ such that $P[\max\{|U|,|V|\} > u] \leq \varepsilon$. By monotonicity of $\mathscr{C}_{L}^{h}$,  and with $\mathbf{h}_u^{\pm} = (h^{\pm}_{u,y})_{y\in \mathbb{Z}^d} $, $h^{\pm}_{u,y} \stackrel{\text{def.}}{=} h \pm p_{yx}u$, so that $\mathbf{h}_u^{-} < \mathbf{h}_u^{+}$, this yields
\begin{equation}\label{eq:inf_bd1}
 I_{\theta}(\mathscr{C}_{L}^{h},x) \leq \mathbb{P}_{\theta,x}[\mathscr{C}_{L}^{\mathbf{h}_u^{-}}(\varphi)] -   \mathbb{P}_{\theta,x}[\mathscr{C}_{L}^{\mathbf{h}_u^{+}}(\varphi)] + \varepsilon =  \mathbb{P}_{\theta,x}[\mathscr{C}_{L}^{\mathbf{h}_u^{-}} \setminus \mathscr{C}_{L}^{\mathbf{h}_u^{+}}] + \varepsilon.
\end{equation}
We introduce the following notion of \textit{pivotal zone}. For an increasing event $A \subset \Omega^{\mathbb{Z}^d}$ and a configuration $\omega$, we say that $K\subset \mathbb{Z}^d$ is a pivotal zone for $A$ in the configuration $\omega$ if $\omega^K \in A$ but $\omega_K \notin A$, and $1_A(\omega^U) = 1_A(\omega_U)$ for all $U \subsetneq K$ (thus $K$ is of minimal size with this property). Accordingly, we define the event $\{ K \text{ is a pivotal set for $A$}\}=\{ \omega ; \,  \omega^K \in A,\, \omega_K \notin A,\, 1_A(\omega^U)= 1_A(\omega_U) \text{ for }U \subsetneq K\}$. For singletons, the second property is obsolete, and one recovers the familiar notion of pivotal sites (see for instance \cite{Gr99}). 

Next, we let $\mathcal{L}^u = \{ y \in \mathbb{Z}^d; \, h^{-}_{u,y} \leq \varphi_y <  h^{+}_{u,y} \}$. Observe that the configurations $\xi^{\mathbf{h}_u^{+}}$ and $\xi^{\mathbf{h}_u^{-}}$ (recall \eqref{eq:def_conf} for notation) agree precisely everywhere outside the (random) set $\mathcal{L}^u$. Moreover, $\mathcal{L}^u$ coincides with the set of sites, open in the configuration $\xi^{\mathbf{h}_u^{-}}$, that one needs to close when raising the level from $\mathbf{h}_u^{-}$ to $\mathbf{h}_u^{+}$. Therefore, 
$$
\mathscr{C}_{L}^{\mathbf{h}_u^{-}} \setminus \mathscr{C}_{L}^{\mathbf{h}_u^{+}} = \{ \mathcal{L}^u \text{ contains a pivotal zone for } \mathscr{C}_{L}^{\mathbf{h}_u^{-}} \} 
$$
(using again monotonicity of  $\mathscr{C}_{L}$). There can typically be more than one pivotal zones in a given configuration, and we denote by $\mathcal{P}^u_L$ the \textit{smallest} (in a given deterministic ordering of the subsets of $B_{2L}$) pivotal zone for $\mathscr{C}_{L}^{\mathbf{h}_u^{-}}$ inside $\mathcal{L}^u$ (on the event $\{ \mathcal{L}^u \text{ contains a pivotal zone for } \mathscr{C}_{L}^{\mathbf{h}_u^{-}} \} $, and otherwise define $\mathcal{P}^u_L=\emptyset$). Returning to \eqref{eq:inf_bd1}, we obtain
\begin{equation}\label{eq:inf_bd2}
I_{\theta}(\mathscr{C}_{L}^{h},x) \leq \mathbb{P}_{\theta,x}[\mathcal{P}^u_L \neq \emptyset] + \varepsilon.
\end{equation}
We first investigate the quantity $\mathbb{P}_{\theta,x}[\mathcal{P}^u_L \cap B(x,3R)^c \neq \emptyset]$, for $L \geq 10R$, separately, and bound
\begin{equation*}
\begin{split}
\mathbb{P}_{\theta,x}[\mathcal{P}^u_L \cap B(x,3R)^c \neq \emptyset] 
&\leq \mathbb{P}_{\theta,x}[ \mathcal{P}^u_L \neq \emptyset, \, \text{dist}(\mathcal{P}^u_L,x) \geq R] + \mathbb{P}_{\theta,x}[\text{diam}(\mathcal{P}^u_L) \geq 2R] \\[0.4em]
&\leq 2 \cdot \mathbb{P}_{\theta,x}[\mathcal{L}^u \cap B(x,R-1)^c \neq \emptyset] \\
&\leq 2\cdot  \sum_{k \geq R}|S(0,k)| \sup_{y \in S(x,k)} \mathbb{P}_{\theta,x}[h^{-}_{u,y} \leq \varphi_y <  h^{+}_{u,y}]\\
&\leq c \sum_{k \geq R} k^{d-1} u   \sup_{y \in S(x,k)} p_{yx}.
\end{split}
\end{equation*}
Recall that the quantity $p_{yx}= P^y_{\theta}[\widetilde{H}_x < \infty]$ decays exponentially in $|y-x|$, thus, by choosing $R=R(\varepsilon)$ sufficiently large, one infers that
\begin{equation}\label{eq:inf_bd4}
\mathbb{P}_{\theta,x}[\mathcal{P}^u_L \cap B(x,3R)^c \neq \emptyset] \leq \varepsilon, \text{ for all } L \geq 10 R(\varepsilon).
\end{equation}
With \eqref{eq:inf_bd4}, coming back to \eqref{eq:inf_bd2}, it follows that
\begin{equation*}
I_{\theta}(\mathscr{C}_{L}^{h},x) \leq \mathbb{P}_{\theta,x}[\mathscr{C}_{L}^{\mathbf{h}_u^{-}}, \,  \emptyset \neq \mathcal{P}^u_L \subset B(x,3R) ] + 2\varepsilon.
\end{equation*}
Now, the occurrence of the event $\mathscr{C}_{L}^{\mathbf{h}_u^{-}}$ and the simultaneous existence of a pivotal zone for this event contained in $B(x,3R)$ imply that $\partial_{\text{out}}B(x,3R)$ is connected to $B_L$ and $S_{2L}$ by two nearest-neighbor paths inside the level-set $E^{\geq \mathbf{h}_u^{-}}$ which do not communicate.
In particular, due to the geometry of the event $\mathscr{C}_{L}^{\mathbf{h}_u^{-}}$, regardless of the position of $x$ inside $B_{2L}$, at least one of these two paths will exit $B(x,L/2)$. All in all, we have thus obtained
\begin{equation} \label{eq:inf_bd6}
I_{\theta}(\mathscr{C}_{L}^{h},x) \leq \mathbb{P}_{\theta,x}[\partial_{\text{out}}B(x,3R) \stackrel{\geq \mathbf{h}_u^{-}}{\longleftrightarrow} S(x,L/2)] + 2\varepsilon \text{ for all $L\geq 10 R(\varepsilon)$}.
\end{equation}
We now show that the probability on the right-hand side of \eqref{eq:inf_bd6} is small as $L\to \infty$. Intuitively, this should be very clear: indeed $\mathbf{h}_u^{-}$ is very close to the constant field with value $h$ outside $B(x,R)$ and $h = h_{**} - 2\ell \geq h_{**}-2\delta$ is subcritical by assumption, cf. \eqref{eq:delta_subcrit} (also neglecting the extra killing at $x$ present under $\mathbb{P}_{\theta,x}$). To make this precise, let us abbreviate $\mathscr{C}_{R,L}^{\mathbf{h}} = \{ \partial_{\text{out}}B(x,3R) \stackrel{\geqslant \mathbf{h}}{\longleftrightarrow} S(x,L/2) \}$, for $\mathbf{h}\in \mathbb{R}^{\mathbb{Z}^d}$. Since $h > h_*$, it follows that $\lim_{L\to\infty} \mathbb{P}_{\theta}[\mathscr{C}_{R,L}^{h}] = 0$ (R is fixed). Moreover,
$$
\mathbb{P}_{\theta}[\mathscr{C}_{R,L}^{h}] \geq \mathbb{E}_{\theta}[ \mathbb{P}_{\theta}[\mathscr{C}_{R,L}^{h}\, | \, \varphi_x] 1\{ \varphi_x \geq 0\}] \stackrel{\eqref{EQ:st_phi_cond_exps}}{\geq}  \mathbb{P}_{\theta,x}[\mathscr{C}_{R,L}^{h}]/2,
$$
for all $L \geq 10R(\varepsilon)$, and therefore $\lim_{L\to\infty} \mathbb{P}_{\theta,x}[\mathscr{C}_{R,L}^{h}] = 0$ as well. Finally, in order to take care of the small perturbation in the level, write
\begin{equation*}
\begin{split}
\mathbb{P}_{\theta,x}[\mathscr{C}_{R,L}^{\mathbf{h}_u^{-}}] 
&\leq \mathbb{P}_{\theta,x}[\mathscr{C}_{R,L}^{h}] +\mathbb{P}_{\theta,x}[\{ y \in \mathbb{Z}^d; \, h^{-}_{u,y} \leq \varphi_y < h \} \cap B(x,3R)^c \neq \emptyset] \\
&\leq \mathbb{P}_{\theta,x}[\mathscr{C}_{R,L}^{h_{**}- \delta}] +\mathbb{P}_{\theta,x}[\mathcal{L}^u \cap B(x,3R)^c \neq \emptyset].
\end{split}
\end{equation*}
The first term goes to zero as $L \to \infty$, and is therefore smaller than $\varepsilon$, for all $L \geq L_0(\delta,\varepsilon)$, and so is the second one, by the same calculation as the one leading to \eqref{eq:inf_bd4}. Substituting into \eqref{eq:inf_bd6} yields that $I_{\theta}(\mathscr{C}_{L}^{h},x) \leq 4\varepsilon$, for all $L \geq L_0(\delta, \varepsilon)$, $x\in B_{2L}$ and $h \in [h_{**}-2\delta, h_{**}-\delta]$, and \eqref{eq:inf_bound_BIS} follows.
\hfill $\square$

\medskip

With Lemma \ref{L:inf_bound_BIS} proved, \eqref{EQ:MAINbis} follows, as explained above. This concludes the proof of Proposition \ref{P:NO_SPEED}.
\end{proof}

We proceed with the proof of our main result, Theorem \ref{T:MAIN}. As alluded to above, this involves working with periodic boundary conditions, in order to obtain the polynomial convergence of Theorem \ref{T:MAIN} for the probability of a well-chosen translation invariant event, related to the quantity $p_{\theta,L}(\cdot)$ we are after. This result is then translated back to the crossing event of interest, first under periodic boundary conditions, and then under the usual ones. 

\bigskip

\noindent \textit{Proof of Theorem \ref{T:MAIN}}. Let $\theta \in (0,1)$ and $M= M(\theta) = 1 + |h_{**}(\theta)| \, (>0)$. By monotonicity of $p_{\theta, L}(\cdot)$, in order to prove \eqref{EQ:MAIN}, it suffices to show
\begin{equation}\label{eq:MAIN_goal}
p_{\theta, L}(h_{**}(\theta)- \delta) \geq 1 - c(\theta, \delta)L^{-\varepsilon(\theta, \delta)}, \text{ for all $L \geq 1$ and $0 <\delta < 1$}.
\end{equation}
Our choice of parameters implies in particular that
\begin{equation} \label{eq:choice_interval}
h_{**}(\theta)- \delta \in (-M(\theta), M(\theta)) \stackrel{\mathrm{def.}}{=}I(\theta), \text{ for all $0 <\delta < 1$}.
\end{equation}
For later purposes, we also fix a parameter 
\begin{equation}\label{eq:ell}
\ell=100.
\end{equation}
The proof operates simultaneously at three scales,
\begin{equation*}
L <  \ell L < \ell^2 L \stackrel{\mathrm{def.}}{=}  \overline{L},
\end{equation*}
where $L$ refers to the quantity appearing in \eqref{eq:MAIN_goal}. In what follows, it will be convenient to identify the vertex set of the torus $\mathbb{T}_{\overline{L}} = (\mathbb{Z}/ 2\overline{L}\mathbb{Z})^d$ with $(\mathbb{Z}\cap [-\overline{L}, \overline{L}))^d  \subset \mathbb{Z}^d$ (and add an edge between all pairs  of corresponding vertices on opposite faces of the box). On $\Omega^{\mathbb{T}_{\overline{L}}}$, we consider the event
\begin{equation}\label{eq:event_A_ti}
\mathcal{A}_{L}=\bigcup_{x\in\mathbb{T}_{\overline{L}}}\{B(x, \ell L)\longleftrightarrow S(x,2\ell L)\}
\end{equation}
(the boxes in question refer to $\ell^{\infty}$-balls on the torus $\mathbb{T}_{\overline{L}}$). As will become apparent below, cf. the discussion leading to \eqref{eq:renorm} and Figure \ref{F:pic_tor}, the reason for introducing the parameter $\ell$ is essentially the following. On $\mathcal{A}_L$, regardless of the particular crossing event $\{B(x, \ell L)\longleftrightarrow S(x,2\ell L)\}$ guaranteeing its occurrence, one can always ensure that a similar crossing event at scale $L$ emanates from one of roughly $\ell^{2d}$ boxes of sidelength $L$ paving the torus $\mathbb{T}_{\overline{L}}$. In particular, the number of such boxes is independent of $L$. 

We define the (decreasing) function 
\begin{equation}
q_{\theta, L}(h)=\mathbb{P}_{\theta}^{   \overline{L}} [\mathcal{A}_{L}^{h}],\text{ for $L \geq 1$ and  \ensuremath{h\in I(\theta)}} \label{EQ:def_q}
\end{equation}
(recall \eqref{eq:events_levset} and \eqref{EQ:notation} for notation). Our first aim is to derive a result similar to \eqref{eq:MAIN_goal}
for the function $q_{\theta, L}(\cdot)$.  Clearly, the event $\mathcal{A}_{L}$ is translation
invariant. Hence, the differential inequality \eqref{EQ:diff_ineq_trans_inv} yields 
\[
-\frac{dq_{\theta, L}(h)}{dh}\geq c(\theta) \cdot q_{\theta, L}(h)(1-q_{\theta, L}(h)) \log L, \text{ for all $L\geq1$ and $h\in I(\theta)$}.
\]
Observing that $\frac{d}{dx}\log(\frac{x}{1-x})=\frac{1}{x(1-x)}$, for all $x \in (0,1)$, and
integrating the previous inequality between arbitrary levels $h<h'$,
with $h,h'\in I(\theta)$, one obtains 
\begin{equation}
\frac{q_{\theta, L}(h)}{1-q_{\theta, L}(h)}\geq\frac{q_{\theta, L}(h')}{1-q_{\theta, L}(h')}L^{c(\theta)(h'-h)},\text{ for all \ensuremath{L\geq1}}.\label{EQ:diff_ineq_int}
\end{equation}
Next, we derive a lower bound for $q_{\theta, L}(\cdot)$ when $h$ is in
the vicinity of $h_{**}$. The latter quantity is defined with respect
to the law $\mathbb{P}_{\theta}$ rather than $\mathbb{P}_{\theta}^{ \overline{L}}$, and the following lemma allows us to change boundary conditions for
sufficiently localized events at small cost. Recall that we identify $\mathbb{T}_{\overline{L}}$ with $(\mathbb{Z}\cap [-\overline{L}, \overline{L}))^d$.
\begin{lem}
\label{L:change_bc} $(\theta \in (0,1))$

\medskip
\noindent For all $L\geq1$, all increasing events $A\subset \Omega^{B(0,\overline{L})}$ measurable with respect to $\sigma (Y_x ; \, x \in B(0, 10 \ell L))$, and all $h\in\mathbb{R}$, 
\begin{equation}
|\mathbb{P}_{\theta}^{\overline{L}}[A^{h}]-\mathbb{P}_{\theta}[A^{h}]|\leq c(\theta)e^{-c'(\theta)\cdot L}.\label{EQ:change_bc}
\end{equation}
\end{lem}
\begin{proof}
We use the domain Markov property to show that, upon conditioning
on 
\begin{equation} \label{eq:def_K_L}
\text{the field on }K_{\overline{L}}=\partial_{\text{int}}B(0,\lceil \overline{L}/2 \rceil),
\end{equation} 
the effect on the
field inside $B(0,\ell L)$ is very small, regardless of the boundary
condition. Specifically, let $g^{\overline{L}}_{\theta, K_{\overline{L}}}(\cdot, \cdot)$ denote the Green function of a random walk on $\mathbb{T}_{\overline{L}}$ killed uniformly at every step and when exiting $\mathbb{T}_{\overline{L}}\setminus K_{\overline{L}}$ (we include the superscript $\overline{L}$ to distinguish it from the corresponding Green function $g_{\theta, K_{\overline{L}}}(\cdot, \cdot)$ on $\mathbb{Z}^d$). First, observe that $K_{\overline{L}}$ has a screening effect, in the sense that $g^{\overline{L}}_{\theta, K_{\overline{L}}}(x,y)=g_{\theta, K_{\overline{L}}}(x,y)$
for all $x,y\in B(0,\ell^{2}L/2)$. Thus, ${\varphi}_{|_{B(0, 10 \ell L)}}$ has the same law under ${\mathbb{P}}_{\theta, K_{\overline{L}}}^{ \overline{L}}$ as under $\mathbb{P}_{\theta, K_{\overline{L}}}$, cf. \eqref{eq:ell}, and in particular, 
\begin{equation}\label{eq:change_bc_A}
{\mathbb{P}}_{\theta, K_{\overline{L}}}^{  \overline{L}}[A_{L}^{h}({\varphi})]={\mathbb{P}}_{\theta, K_{\overline{L}}}[A_{L}^{h}({\varphi})],
\end{equation}
for all $h\in\mathbb{R}$ and $L\geq1$, by the measurability assumption on $A$. Now, by a union bound and elementary Gaussian
tail estimate, we have that $\mathbb{P}_{\theta}^{ \overline{L}}[\max_{K_{\overline{L}}}|\varphi|\geq L]\leq C(\theta)e^{-c(\theta)L^{2}}$,
for all $L\geq1$ (and a similar bound with $\mathbb{P}_{\theta}$ in place of $\mathbb{P}_{\theta}^{ \overline{L}}$). Therefore, by virtue of Lemma \ref{L:cond_exps} (with $U = \emptyset$ and $K=K_L$), and using \eqref{eq:change_bc_A},
\[
\mathbb{P}_{\theta}^{ \overline{L}}[A^{h}]=\mathbb{E}_{\theta}^{  \overline{L}}\Big[\widetilde{\mathbb{P}}_{\theta, K_{\overline{L}}}[A^{h}(\widetilde{\varphi}_{\cdot}+\mu_{\cdot}(\varphi))]\cdot 1\{\max_{y\in K_{\overline{L}}}\varphi_{y}\leq L \}\Big]+o(e^{-L}), \text{ as $L\to\infty$. }
\]
Here, and in accordance with \eqref{ind+cond_exps}, $\widetilde{\mathbb{P}}_{\theta, K_{\overline{L}}}$ stands for a copy of ${\mathbb{P}}_{\theta, K_{\overline{L}}}$ governing the field $\widetilde{\varphi}$, which is independent of $\varphi$, and $\mu_{\cdot} = \mu_{\cdot}(\varphi)$ is as defined in \eqref{mu}. Moreover, on the event $\{\max_{K_{\overline{L}}}\varphi\leq L\}$, and for arbitrary $x\in B(0,10\ell L)$,
one obtains, for all $L\geq1$, with $P^{x}_{\theta}$ denoting the law of the random walk on $\mathbb{T}_{\overline{L}}$ started at $x$ and killed at rate $\theta$,
\begin{equation*}
\begin{array}{rcl}
\mu_{x} \hspace{-1ex}& \stackrel{\eqref{mu}}{=} & \displaystyle  \hspace{-1ex} \sum_{y\in K_{\overline{L}}}P^{x}_{\theta}[H_{K_{\overline{L}}}<\infty,X_{H_{K_{\overline{L}}}}=y]\cdot \varphi_{y}\\
 \hspace{-1ex}& \leq &  \hspace{-1ex} \displaystyle P^{x}_{\theta}[H_{K_{\overline{L}}}<\infty]\cdot\max_{y\in K_{\overline{L}}}\varphi_{y}\leq c(\theta)e^{-c'(\theta)\cdot L}\stackrel{\text{def.}}{=}\varepsilon_{L},
 \end{array}
\end{equation*}
since $|K_{\overline{L}}|\leq CL^{d-1}$ and $\text{dist}(x,K_{\overline{L}})>L$, cf. \eqref{eq:ell} and \eqref{eq:def_K_L}. By monotonicity
of $A$, this yields 
\begin{equation}
\mathbb{P}_{\theta}^{ \overline{L}}[A^{h}]\leq \widetilde{\mathbb{P}}_{\theta, K_{\overline{L}}}[A^{h-\varepsilon_{L}}(\widetilde{\varphi})]+c(\theta)e^{-c'(\theta)L},\text{ for all \ensuremath{L\geq1} and $h \in \mathbb{R}$.}\label{EQ:bc_estimate1}
\end{equation}
By a similar argument, one also has the lower bound 
\begin{equation}
\begin{split}\mathbb{P}_{\theta}[A^{h}] & \geq\mathbb{P}_{\theta}[A^{h},\min_{y\in K_{\overline{L}}}\varphi_{y}\geq-L]\\
 & \geq\widetilde{\mathbb{P}}_{\theta, K_{\overline{L}}}[A^{h+\varepsilon_{L}}(\widetilde{\varphi})] \cdot \mathbb{P}_{\theta}[\min_{y\in K_{\overline{L}}}\varphi_{y}\geq-L]\geq \widetilde{\mathbb{P}}_{\theta, K_{\overline{L}}}[A^{h+\varepsilon_{L}}(\widetilde{\varphi})]-c(\theta)e^{-c'(\theta)L},
\end{split}
\label{EQ:bc_estimate2}
\end{equation}
for all $L\geq1$ and $h \in \mathbb{R}$. Together, (\ref{EQ:bc_estimate1})
and (\ref{EQ:bc_estimate2}) yield 
\begin{align*}
\mathbb{P}_{\theta}^{  \overline{L}}[A^{h}]-\mathbb{P}_{\theta}[A^{h}] & \leq {\mathbb{P}}_{\theta, K_{\overline{L}}}[A^{h-\varepsilon_{L}}({\varphi})]-{\mathbb{P}}_{\theta, K_{\overline{L}}}[A^{h+\varepsilon_{L}}({\varphi})]+c(\theta)e^{-c'(\theta)L}\\
 & \leq {\mathbb{P}}_{\theta, K_{\overline{L}}}[\, |\varphi_{x}-h| < \varepsilon_{L},\text{ for some } x \in B(0,10\ell L)]+c(\theta)e^{-c'(\theta)L}\\
 & cL^d \cdot 2\varepsilon_{L} \sup_{x \in B(0,10\ell L)} (2\pi g_{\theta, K_{\overline{L}}}(x,x))^{-1/2} +c(\theta)e^{-c'(\theta)L} \\
 &\leq c''(\theta)e^{-c'(\theta)L},
\end{align*}
for all $L\geq1$ and $h \in \mathbb{R}$, where the second line follows since the configurations $\xi^{h- \varepsilon_L}(\varphi)$ and $\xi^{h+ \varepsilon_L}(\varphi)$ coincide otherwise, implying in particular that $1_{A^{h-\varepsilon_{L}}} = 1_{A^{h+\varepsilon_{L}}}$, and the last line because $g_{\theta, K_{\overline{L}}}(x,x) \geq 1$ for all $x$, thus yielding a uniform (in $L$) upper bound on the marginal densities of the killed Gaussian free field. The above proof can be repeated with the role of $\mathbb{P}_{\theta}^{  \overline{L}}$
and $\mathbb{P}_{\theta}$ interchanged, and (\ref{EQ:change_bc}) follows. 
\end{proof}

We return to the proof of \eqref{eq:MAIN_goal}. Let $\delta \in (0,1)$. Lemma \ref{L:change_bc} yields a uniform lower bound for $q_{\theta, L}(h_{**}-\delta/2)$
as $L$ becomes large. Indeed, first note that by definition of $h_{**}$, see Theorem \nolinebreak\ref{T:h_**_properties}, there exists $\widetilde{L}_{0}(\delta, \theta)$
such that, for all $L\geq \widetilde{L}_{0}(\delta, \theta)$ and $h\leq h_{**}-\delta/2$,
$p_{\theta, L}(h)\geq \gamma(\theta) \ (>0)$. On account of \eqref{EQ:change_bc}, this implies that there exists $L_{0}(\delta, \theta)\geq\widetilde{L}_{0}(\delta, \theta)$
such that 
\begin{equation}
\mathbb{P}_{\theta}^{ \overline{L}}[B_{L}\stackrel{\geqslant h}{\longleftrightarrow}\partial_{\mathrm{int}}B_{2L}]\geq \gamma(\theta)/2,\text{ for all \ensuremath{h\leq h_{**}-\delta/2} and \ensuremath{L\geq L_{0}(\delta,\theta)},}\label{EQ:p_L_perLB1}
\end{equation}
where $B_L$ is short for $B(0,L)$, $L \geq 1$. In particular, by definition of $\mathcal{A}_{L}$, this yields, for $L\geq L_{0}(\delta,\theta)$, 
\[
q_{\theta, L}(h_{**}-\delta/2)\stackrel{\eqref{EQ:def_q}}{\geq}\mathbb{P}_{\theta}^{   \overline{L}}[B_{\ell L}\stackrel{\geqslant h_{**}-\delta/2}{\longleftrightarrow}S_{2\ell L}]\stackrel{\eqref{EQ:p_L_perLB1}}{\geq}\gamma(\theta)/2,
\]
as desired. Returning to (\ref{EQ:diff_ineq_int}), one obtains, setting
$h=h_{**}-\delta$ and $h'=h_{**}-\delta/2$, which satisfy $h,h' \in I(\theta)$ by \eqref{eq:choice_interval} and the choice of $M(\theta)$, that for all $L\geq L_{0}(\delta,\theta)$,
\begin{align*}
\frac{1}{1-q_{\theta, L}(h_{**}-\delta)}\geq\frac{q_{\theta, L}(h_{**}-\delta)}{1-q_{\theta, L}(h_{**}-\delta)} \stackrel{\eqref{EQ:diff_ineq_int}}{\geq} \frac{q_{\theta, L}(h_{**}-\delta/2)}{1-q_{\theta, L}(h_{**}-\delta/2)}L^{c(\theta)\cdot\delta}\geq\frac{\gamma(\theta)}{2}L^{c'(\theta, \delta)},
\end{align*}
and after rearranging 
\begin{equation}
q_{\theta, L}(h_{**}-\delta)\geq1-c(\theta)L^{-c'(\delta,\theta)},\text{ for all \ensuremath{L\geq L_{0}(\delta,\theta)}}.\label{EQ:q_LB}
\end{equation}
This bound is similar to the statement (\ref{eq:MAIN_goal})
we are trying to prove, but with the function $q_{\theta, L}(\cdot)$ in place
of $p_{\theta, L}(\cdot)$. In order to make the necessary replacement, we
need to do two things: first, to pass from $\mathcal{A}_{L}$, which
is a union of crossing events at scale $\sim \ell L$, to a crossing event
between two \textit{fixed} concentric boxes of size $\sim L$, and second, to change the boundary conditions.

To tackle the former issue, we use an immediate consequence of the FKG-inequality, often called ``square-root trick'' (see for instance \cite{Gr99}, p.289 for a proof). If $A_1,\dots, A_n$, are all increasing measurable subsets of $\mathbb{R}^\mathcal{G}$
\begin{equation}\label{eq:sqrt_trick}
\sup_{1\leq i \leq n} \mathbb{P}_{\theta}^{\mathcal{G}}[A_i] \geq 1 - \Big(1- \mathbb{P}_{\theta}^{\mathcal{G}}\Big[ \bigcup_{i=1}^n A_i \Big] \Big)^\frac{1}{n}.
\end{equation}
This is typically used to ensure that $\sup_{1\leq i \leq n} \mathbb{P}_{\theta}^{\mathcal{G}}[A_i]$ is close to one, provided the probability of the union is (and $n$ is not too large).
\begin{figure}[h!]
  \centering 
  \includegraphics[scale=0.9]{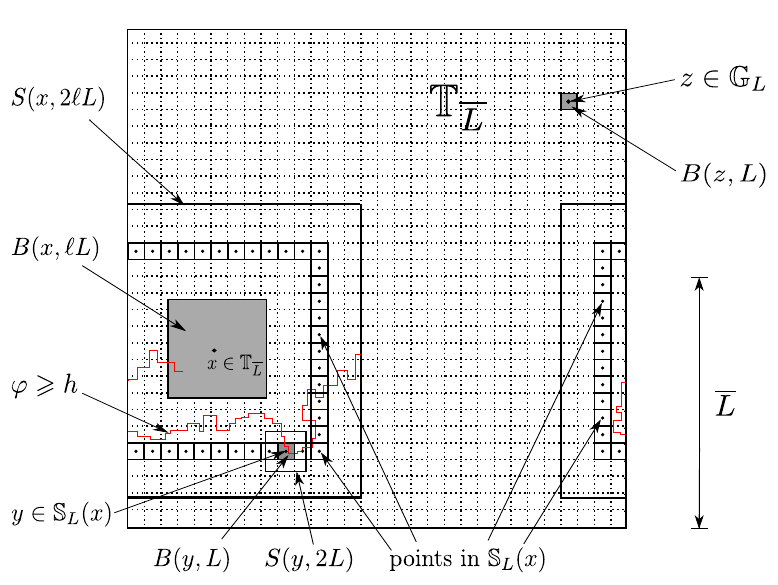}
  \caption{The torus $\mathbb{T}_{\overline{L}}$ is paved by a \textit{fixed} (i.e. independent of $L$) number of boxes of radius $L$, indicated by the dotted lines. The centers of these boxes form the grid $\mathbb{G}_L$. Regardless of the exact location of $x$ on the torus $\mathbb{T}_{\overline{L}}$, if the event $\{ B(x,\ell L) \stackrel{\geqslant h}{\longleftrightarrow} S(x,2\ell L) \}$ occurs (as evidenced by the red path), then so must $\{ B(y, L) \stackrel{\geqslant h}{\longleftrightarrow} S(y,2 L) \}$, for at least one point $y \in \mathbb{G}_L$.}
  \label{F:pic_tor}
\end{figure}
In order to apply \eqref{eq:sqrt_trick} in the present context, we consider the set
\[
\mathbb{G}_{L}=(\mathbb{Z}\cap [-\overline{L}, \overline{L}))^d \cap L\mathbb{Z}^{d} \ (\subset \mathbb{T}_{\overline{L}})
\]
(recall that $(\mathbb{Z}\cap [-\overline{L}, \overline{L}))^d \subset \mathbb{Z}^d$ is identified with the vertex set of $\mathbb{T}_{\overline{L}}$). We view the set $\mathbb{G}_{L}$ as a grid (of mesh size $L$) on the torus $\mathbb{T}_{\overline{L}}$. Note
that $|\mathbb{G}_{L}| = (2\ell^{2})^{d}$, which does not depend on
$L$. Now, suppose $L \geq 100$, and given $x \in \mathbb{T}_{\overline{L}}$, consider the \textit{shell} 
$$
\mathbb{S}_L(x) = \{ y \in \mathbb{G}_{L} ; \,  B(y,L) \cap S(x, 3\ell L/2) \neq \emptyset \},
$$
see also Figure \ref{F:pic_tor}. The set $\bigcup_{y \in \mathbb{S}_L(x)} B(y,L)$ (part of $\mathbb{T}_{\overline{L}}$) disconnects $B(x, \ell L)$ from $S(x, 2\ell L)$, for all $x \in \mathbb{T}_{\overline{L}}$ and $L \geq 100$. Moreover, $ B(y,2L) \subset B(x, 2\ell L)$ for all $y \in \mathbb{S}_L(x)$. Thus, any nearest-neighbor path joining $B(x, \ell L)$ to $S(x, 2\ell L)$ must intersect $B(y,L)$, for at least one $y \in  \mathbb{S}_L(x)$, and subsequently exit $S(y,2L)$, i.e.
\begin{equation}\label{eq:renorm}
\{B(x, \ell L) \longleftrightarrow S(x, 2\ell L) \} \subset \bigcup_{y \in \mathbb{S}_L(x)} \{B(y, L) \longleftrightarrow S(y, 2 L) \}, \text { for $x \in \mathbb{T}_{\overline{L}}$ and $L \geq 100$,}
\end{equation}
Consequently,
\[
\mathcal{A}_{L}^{h}\stackrel{\eqref{eq:event_A_ti}, \eqref{eq:renorm}}{\subset} \bigcup_{y\in\mathbb{G}_{L}}\{{B}(y,L)\stackrel{\geqslant h}{\longleftrightarrow}{S}(y,2L)\},
\]
for all $h\in\mathbb{R}$, $L \geq 100$. The crucial point is that the right-hand
side is a union over a fixed (i.e. independent of $L$) number of
events, which all have the same probability under $\mathbb{P}_{\theta}^{  \,  \mathbf{p}, \overline{L}}$. An application of \eqref{eq:sqrt_trick} then yields
\begin{align*}
\mathbb{P}_{\theta}^{  \overline{L}}[{B}(0,L)\stackrel{\geqslant h}{\longleftrightarrow}{S}(0,2L)] & \geq1-\Big(1-\mathbb{P}_{\theta}^{  \overline{L}}\Big[\bigcup_{y\in\mathbb{G}_{L}}\{{B}(y,L)\stackrel{\geqslant h}{\longleftrightarrow}{S}(y,2L)\}\Big]\Big)^{1/(2\ell^{2})^{d}}\\
 & \geq1-(1-q_{\theta, L}(h))^{1/(2\ell^{2})^{d}},
\end{align*}
for all $L\geq100$ and $h\in\mathbb{R}$. Hence, on account
of (\ref{EQ:q_LB}), it follows that 
\[
\mathbb{P}_{\theta}^{  \overline{L}}[{B}(0,L)\stackrel{\geqslant h_{**}-\delta}{\longleftrightarrow}{S}(0,2L)]\geq1- c(\theta) L^{-c'(\delta, \theta)}
\]
for all $L\geq 100 \vee L_{0}(\delta, \theta)$. Then, by Lemma \ref{L:change_bc},
we see that $\mathbb{P}_{\theta}^{ \overline{L}}[{B}(0,L)\stackrel{\geqslant h_{**}-\delta}{\longleftrightarrow}{S}(0,2L)]$
can be replaced by $p_{\theta,L}(h_{**}-\delta)  = \mathbb{P}_{\theta}[B(0,L)\stackrel{\geqslant h}{\longleftrightarrow}S(0,2L)]$ upon possibly enlarging  $L_{0}(\delta, \theta)$. Finally \eqref{eq:MAIN_goal} follows by adapting the constant $c(\theta)$ (in a manner depending on $\delta$), as to allow for all $L \geq 1$. This completes the proof of Theorem \ref{T:MAIN}. \hfill $\square$

\medskip

\begin{remark}\label{R:FINAL} 1) The square-root trick in the above proof is reminiscent
of an argument of Bollob\'as and Riordan \cite{BR06} in the context of rectangular side-to-side crossings for percolation on two-dimensional (random) Voronoi tessellations, which was later re-used to derive sharp-threshold results for similar crossing events in the planar
random cluster model, see \cite{GG06}, \cite{BDC12}.

\medskip

\noindent 2) (Generalizations). For clarity of exposition, we have considered the case of symmetric simple random walk, but our results readily generalize e.g. to any conductance model on $\mathcal{G}$, with, say, bounded conductances, and a uniform non-zero killing measure (our setup corresponds to putting unit conductances on all edges of $\mathcal{G}$). Moreover, most of our results, among them, Theorem \ref{T:MAIN}, Proposition \ref{P:dom_inf_ptwise} and Corollary \ref{C:diff_ineq}, continue to hold in dimension $d=2$. However, the availability of additional tools, and in particular, planar duality techniques, might allow for certain simplifications in the proofs. \hfill $\square$

\end{remark}

\noindent{\textbf{Acknowledgements.}} The author thanks Alain-Sol Sznitman for several useful discussions and for his comments on an earlier draft of this manuscript.

\bibliography{rodriguez}
\bibliographystyle{plain}

\newpage{}

\end{document}